\documentclass{amsart}

\usepackage{amssymb}
\usepackage{mathrsfs}
\usepackage{subfigure}
\usepackage{graphicx}

\usepackage[cmtip,all]{xy}

\newtheorem{theorem}{Theorem}[section]
\newtheorem{lemma}[theorem]{Lemma}
\newtheorem{proposition}[theorem]{Proposition}

\theoremstyle{definition}

\theoremstyle{remark}
\newtheorem{remark}[theorem]{Remark}

\numberwithin{equation}{section}

\begin{document}

\title[Strong convergence rate of schemes for SDEs driven by fBms]{Strong convergence rate of Runge--Kutta methods and simplified step-$N$ Euler schemes for SDEs driven by fractional Brownian motions}


\author{Jialin Hong}
\address{LSEC, ICMSEC, Academy of Mathematics and Systems Science, Chinese Academy of Sciences, Beijing 100190, China; School of Mathematical Sciences, University of Chinese Academy of Sciences, Beijing 100049, China}
\curraddr{}
\email{hjl@lsec.cc.ac.cn}
\thanks{Authors are funded by National Natural Science Foundation of China (NO. 91530118, NO. 91130003, NO. 11021101, NO. 91630312 and NO. 11290142)}

\author{Chuying Huang}
\address{LSEC, ICMSEC, Academy of Mathematics and Systems Science, Chinese Academy of Sciences, Beijing 100190, China; School of Mathematical Sciences, University of Chinese Academy of Sciences, Beijing 100049, China}
\curraddr{}
\email{huangchuying@lsec.cc.ac.cn (Corresponding author)}
\thanks{}

\author{Xu Wang}
\address{LSEC, ICMSEC, Academy of Mathematics and Systems Science, Chinese Academy of Sciences, Beijing 100190, China; School of Mathematical Sciences, University of Chinese Academy of Sciences, Beijing 100049, China}
\curraddr{}
\email{wangxu@lsec.cc.ac.cn}
\thanks{}


\date{}

\dedicatory{}


\begin{abstract}
	This paper focuses on the strong convergence rate of both Runge--Kutta methods and simplified step-$N$ Euler schemes for stochastic differential equations driven by multi-dimensional fractional Brownian motions with $H\in(\frac12,1)$. Based on the continuous dependence of both stage values and numerical schemes on driving noises, order conditions of Runge--Kutta methods are proposed for the optimal strong convergence rate $2H-\frac12$. This provides an alternative way to analyze the convergence rate of explicit schemes by adding `stage values' such that the schemes are comparable with Runge--Kutta methods. Taking advantage of this technique, the optimal strong convergence rate of simplified step-N Euler scheme is obtained, which gives an answer to a conjecture in \cite{Deya} when $H\in(\frac12,1)$. Numerical experiments verify the theoretial convergence rate.
\end{abstract}

\maketitle

\section{Introduction}
In this paper, we investigate the strong convergence rate of numerical schemes for the following stochastic differential equation (SDE)
\begin{align}\label{sde}
\begin{split}
dY_t&=V(Y_t)dX_t=\sum_{l=1}^{d}V_l(Y_t)dX_t^l,\quad t\in(0,T],\\
Y_0&=y\in \mathbb{R}^m,
\end{split}
\end{align}
where $X_t=(X^1_t,\cdots,X^d_t)^\top\in\mathbb{R}^d$ with $X^1_t=t$ and $X^2_t,\cdots,X^d_t$ being independent fractional Brownian motions with Hurst parameter $H\in(\frac12,1)$. The  well-posedness is interpreted through Young's integral or fractional calculus (see  \cite{Friz,Lyons,Zahle98RTRF} and references therein) pathwisely.

The fractional Brownian motion (fBm) $\{B^H_t\}_{t\in[0,T]}$ on some probability space $(\Omega,\mathcal{F},\mathbb{P})$ is a centered Gaussian process with continuous sample paths. Its covariance satisfies
\begin{align}\label{covariance}
\mathbb{E}\Big[B^H_sB^H_t\Big]=\frac12\bigg(t^{2H}+s^{2H}-|t-s|^{2H}\bigg),\quad \forall\ s,t\in [0,T],
\end{align}
where $H\in(0,1)$ is the Hurst parameter. The fBm is a semi-martingale and Markovian process only when $H=\frac12$, that is, standard Brownian motion. Otherwise, the process exhibits long-range or short-range dependence when $H>\frac12$ or $H<\frac12$ respectively. It fits better than Markovian ones in models of economics, fluctuations of solids, hydrology and so on, which motivates numerous researches (see e.g. \cite{MJ68SIAMRev}).
However, it brings more obstacles in both the simulation of noises and analysis of optimal strong convergence rate. 

On the one hand, nontrivial covariance causes difficulties in simulating iterated integrals of fBms from Taylor expansion in multi-dimensional case. 
For the standard Brownian case, iterated integrals can be simulated by specific independent and identically distributed Gaussian random variables. 
For the case $H\neq \frac12$, other techniques need to be explored.
An implementable choice is substituting the $N$-level iterated integral of $X$ by $\frac{1}{N!}(\Delta X_k)^N$ directly, $N\ge 2$. The corresponding numerical schemes are called simplified step-$N$ Euler schemes (see \cite{MLMC16,Deya,Wzk,GN09AIHP}). Another way is taking advantage of internal stages values to design Runge--Kutta methods. These methods are derivative free and can be particularly chosen as structure-preserving methods or stability preserving methods (see \cite{GGH15mcom,HLS06mcom,Milstein} and references therein).

On the other hand, without properties of martingale, approaches to analyze the convergence rate for schemes in fractional setting are different from those via the fundamental convergence theorem in standard Brownian case. In \cite{Deya}, authors analyze the modified Milstein scheme, which is called simplified step-$2$ Euler scheme in this paper. They smooth noises by piecewise-linear approximations, i.e., the Wong--Zakai approximations, and obtain the pathwise convergence rate $(H-\frac13)^-$ in H\"older norm for any $H\in(\frac13,1)$. They conjecture that the optimal rate in supremum norm is $2H-\frac12$ based on the strong convergence rate of the L\'evy area of $X$. In \cite{CN2017}, by estimates for the L\'evy area type processes, authors prove the optimal strong convergence rate of Crank-Nicolson scheme is $2H-\frac12$ in general. Namely, denote by $Y^n$ the numerical solution with time step $h=\frac{T}{n}$, then
\begin{align}\label{CN}
\sup_{t\in[0,T]}\big\|Y_t-Y^n_t\big\|_{L^p(\Omega)}\le C h^{2H-\frac12},\quad p\ge 1.
\end{align}
It inspires us a potential approach to gain the optimal global strong convergence rate for the schemes under study without utilizing the Wong--Zakai approximations.

Our main idea is regarding general Runge--Kutta methods as implicit ones determined through the internal stage values. We show that the constructed continuous versions of internal stage values and numerical solutions are continuously dependent on the driving noise $X$. This robustness coincides with the property of the exact solution. Combining the estimates of  internal stage values and that of iterated integrals of $X$ (see \cite{CN2017,HuTaylor}), we obtain order conditions of the optimal strong convergence rate for Runge--Kutta methods.  Namely, if the coefficients of a Runge--Kutta method satisfy \eqref{condition}, then
\begin{align}\label{RKintro}
\bigg\|\sup_{t\in[0,T]}|Y_t-Y^n_t|\bigg\|_{L^p(\Omega)}\le C h^{2H-\frac{1}{2}},\quad p\ge 1.
\end{align}
Notice that condition \eqref{condition} can be satisfied by Crank-Nicolson scheme and we give the $L^p(\Omega)$-estimate for the error in supremum norm. 

Furthermore, we compare simplified step-$N$ schemes with Runge-Kutta methods satisfying condition \eqref{condition}. By means of adding internal stage values to explicit simplified step-$N$ Euler schemes, we express these schemes in an implicit way. This approach leads us to the optimal stong convergence rate $2H-\frac12$ by avoiding the estimation of the Wong--Zakai approximations. Our result gives an answer to the conjecture in \cite{Deya} for $H\in(\frac12,1)$. Numerical experiments are represented to verify this optimal rate. For the case $H\in(\frac13,\frac12)$ where equation \eqref{sde} is understood in the rough path framework (see e.g. \cite{Davie,Friz,Lyons}), the optimal strong convergence rate of simplified step-$2$ Euler scheme is still an open problem. We refer to related works \cite{MLMC16,Firstorder2017} for more details.

The paper is organized as follows. In Section \ref{sec2}, we recall some definitions and results about fractional calculus and fBms. We prove the solvability of implicit Runge--Kutta methods and the continuous dependence of continuous versions of methods under study with respect to driving noises in H\"older semi-norm in Section \ref{sec3}. The order conditions for Runge--Kutta methods are derived for the strong convergence rate $2H-\frac12$ in Section \ref{sec4}. Comparing the simplified step-$N$ Euler schemes with a Runge--Kutta method satisfying \eqref{condition}, we get the same strong convergence rate, which coincides with the conjecture given by \cite{Deya}. Numerical experiments are performed in Section \ref{sec5}.

\section{Preliminaries}\label{sec2}
In this section, we introduce some notations, definitions and results about fractional calculus and fBms. They are essential for us to prove the properties and strong convergence rates of numerical schemes in subsequent sections. We use $C$ as a generic constant which could be different from line to line.
\subsection{Fractional calculus}
Denote by $\mathcal{C}([0,T];\mathbb{R}^d)$ the space of continuous functions from  $[0,T]$ to $\mathbb{R}^d$. For any $f\in \mathcal{C}([0,T];\mathbb{R}^d)$, $0\le s<t \le T$ and $0<\beta\le1$, the $\beta$-H\"older semi-norm of $f$ is defined by 
\begin{align*}
\|f\|_{s,t,\beta}:=\sup\bigg\{\frac{|f_v-f_u|}{(v-u)^\beta},~s\le u<v\le t\bigg\},
\end{align*}
where $|\cdot|$ is the Euclid norm in $\mathbb{R}^d$.
Especially, we use $\|f\|_{\beta}:=\|f\|_{0,T,\beta}$ for short. The H\"older semi-norm can be expressed in an integral form by the Besov-H\"older embedding, which is a corollary from Garsia-Rademich-Rumsey inequality.
\begin{lemma} (\cite[Corollary A.2]{Friz})\label{Besov}
	Let $q>1$, $\alpha\in(\frac{1}{q},1)$ and $f\in \mathcal{C}([0,T];\mathbb{R}^d)$. Then there exists a constant $C=C(\alpha,q)$ such that for all $0\le s < t\le T$, 
	\begin{align*}
	\|f\|^q_{s,t,\alpha-\frac{1}{q}}\le C \int_{s}^{t} \int_{s}^{t} \frac{|f_u-f_v|^q}{|u-v|^{1+q\alpha}}dudv.
	\end{align*}
\end{lemma}

Let $f\in \mathcal{C}([s,t];\mathbb{R})$ be $\beta$-H\"older continuous on $[s,t]\subseteq[0,T]$ with $1/2<\beta<1$ and $g:[s,t]\rightarrow \mathbb{R}$ be a step function defined by  $g_t=g_0\mathbf{1}_{\{0\}}+\sum^{n-1}_{k=0}g^k\mathbf{1}_{(t_k,t_{k+1}]}$ with $s=t_0<t_1<\cdots<t_n=t$. 
The integral of $g$ with respect to $f$ can be defined piecewisely:
\begin{align*}
\int_{s}^{t}g_rdf_r:=\sum_{k=0}^{n-1}g^k(f_{t_{k+1}}-f_{t_k}).
\end{align*}
For any $1/2<\alpha<1$, according to fractional calculus (see e.g. \cite[Section 2]{Zahle98RTRF}), it has the characterization:
\begin{align*}
\int_{s}^{t}g_rdf_r=(-1)^\alpha \int_{s}^{t} D^\alpha_{s+}g_r D^{1-\alpha}_{t-} F_rdr.
\end{align*}
Here $(-1)^\alpha=e^{-i\pi\alpha}$, $F_r:=f_r-f_t$, $D^\alpha_{s+}g_r$ and $D^{1-\alpha}_{t-} F_r$ are fractional Weyl derivatives of the order $\alpha$ and $1-\alpha$ respectively:
\begin{align*}
(D^\alpha_{s+}g)_r:&=\frac{1}{\varGamma(1-\alpha)}\bigg(\frac{g_r}{(r-s)^\alpha}+\alpha
\int_{s}^{r}\frac{g_r-g_u}{(r-u)^{\alpha+1}}du\bigg),\\
(D^{1-\alpha}_{t-}F)_r:&=\frac{(-1)^{1-\alpha}}{\varGamma(\alpha)}\bigg(\frac{F_r}{(t-r)^{1-\alpha}}+(1-\alpha)
\int_{r}^{t}\frac{F_r-F_u}{(u-r)^{2-\alpha}}du\bigg).
\end{align*}

\subsection{A priori estimate for the solution and iterated integrals}

In the sequel, we denote by $\mathcal{C}_b^N(\mathbb{R}^m;\mathbb{R}^{M})$ the space of bounded and $N$-times continuously differentiable functions $V:\mathbb{R}^m \rightarrow \mathbb{R}^{M}$ with bounded derivatives.

The following lemma shows the well-posedness of \eqref{sde}, which means that the solution is continuously dependent on the driving noises in H\"older semi-norm, where almost all sample paths of $X$ are $\beta$-H\"older continuous for any $\beta\in(0,H)$. In the next section, we will show that the numerical schemes we consider could inherit a similar property. 
\begin{lemma}(see e.g. \cite[Theorem 10.14]{Friz})\label{well}
	If $V\in \mathcal{C}^{1}_b(\mathbb{R}^m;\mathbb{R}^{m\times d})$ and $1/2<\beta<H$, then there exists a unique solution of \eqref{sde} satisfying almost surely that
	\begin{align*}
	\|Y\|_{\beta}&\le C(V,\beta,T)\max\bigg\{\|X\|_{\beta},\|X\|^{1/\beta}_{\beta}\bigg\},\\
	\|Y\|_{\infty}&\le |y|+C(V,\beta,T)\max\bigg\{\|X\|_{\beta},\|X\|^{1/\beta}_{\beta}\bigg\},\\
	\end{align*}
	where $\|Y\|_{\infty}:=\sup\big\{|Y_u|,~0\le u\le T\big\}.$
	Moreover, for some $C_0>0$ and $0\le s<t\le T$ such that $\|X\|_\beta|t-s|^\beta\le C_0$, the estimate can be improved to
	\begin{align*}
	\|Y\|_{s,t,\beta}&\le C(V,\beta,T,C_0)\|X\|_{\beta}.\\
	\end{align*}
\end{lemma}

To get the strong convergence rate of numerical schemes, we recall some results from \cite{CN2017,HuTaylor}. For a numerical scheme, we apply the uniform partition of the interval $[0,T]$ with step size $h=\frac{T}{n}$, $n\in\mathbb{N_+}$ and  denote $t_k=kh$, $k=0,\cdots,n$.

\begin{lemma}(see \cite{CN2017,HuTaylor})\label{multi}
	Let $X^1_t=t$ and $X^2_t,\cdots,X^d_t$ be independent fBms with $H>\frac{1}{2}$. Then for any $n\in\mathbb{N}_+$, it holds for any $0\le t_i<t_j\le T$ and $p\ge 1$ that
	\begin{align}
	\bigg\|\sum_{k=i}^{j-1}\bigg[\int_{t_k}^{t_{k+1}}\int_{t_k}^{s}dX^{3}_udX^2_s
	-\int_{t_k}^{t_{k+1}}\int_{s}^{t_{k+1}}  dX^{3}_udX^2_s\bigg]\bigg\|_{L^p(\Omega)}
	&\le C|t_j-t_i|^{\frac{1}{2}} h^{2H-\frac{1}{2}},\label{levyB}\\
	\bigg\|\sum_{k=i}^{j-1}\bigg[\int_{t_k}^{t_{k+1}}\int_{t_k}^{s}dX^1_udX^2_s
	-\int_{t_k}^{t_{k+1}}\int_{s}^{t_{k+1}}  dX^1_udX^2_s\bigg]\bigg\|_{L^p(\Omega)}
	&\le C|t_j-t_i|^{\frac{1}{2}} h^{H+\frac12}\label{levyt},
	\end{align}
	where $C=C(p)$ above is independent of $n$.
	Moreover, for any $l_1,\cdots,l_{N'}\in\{1,\cdots,d\} $, it holds that 
	\begin{align}
	\bigg\|\sum_{k=i}^{j-1}\int_{t_k}^{t_{k+1}}\int_{t_k}^{u_1}\cdots \int_{t_k}^{u_{N'-1}} dX^{l_{N'}}_{u_{N'}}\cdots dX^{l_2}_{u_2}dX^{l_1}_{u_1}\bigg\|_{L^p(\Omega)}
	\le C|t_j-t_i|^{\frac{1}{2}} h^{r}\label{multiple},
	\end{align}
	where $r=N''H+N'-N''-1$ when $N''=\sharp\{l_i:l_i\neq 1\}$ is even, $r=N''H+N'-N''-H$ when $N''$ is odd, and $C=C(p)$ is independent of $n$.
	
\end{lemma}


Lemma \ref{multi} shows estimates for the L\'evy area type processes and multiple integrals of $X$. In particular, if $N'=N''=2$, then $r= 2H-1<2H-\frac{1}{2}$. This implies that the convergence rate of the $2$nd-level iterated integrals of $X$ in the form of \eqref{levyB} is higher than that in the form of \eqref{multiple}. 


\begin{lemma}(\cite[Proposition 8]{HuTaylor})\label{trans}
	Let $f$ be $\beta$-H\"older continuous stochastic process in $L^{2p}(\Omega)$ with $\frac{1}{2}<\beta<H$ and $p\ge 1$, i.e., 
	\begin{align*}
	\sup\bigg\{\frac{\|f_v-f_u\|_{L^{2p}(\Omega)}}{(v-u)^\beta},~0\le u<v\le T\bigg\}<\infty.
	\end{align*}
	If a sequence of stochastic processes $\{g_n\}_{n\in\mathbb{N}_+}$ satisfies $g_n(t_i)=\sum_{k=0}^{i-1}\xi_{n,k}$ and 
	\begin{align*}
	\|g_n(t_j)-g_n(t_i)\|_{L^{2p}(\Omega)}\le C|t_j-t_i|^{\frac{1}{2}},\quad \forall ~0\le t_i<t_j\le T,
	\end{align*}
	then 
	\begin{align*}
	\bigg\|\sum_{k=i}^{j-1}f_{t_k}\xi_{n,k} \bigg\|_{L^{p}(\Omega)}\le C|t_j-t_i|^{\frac{1}{2}},\quad \forall ~0\le t_i<t_j\le T.
	\end{align*}
	Constants $C=C(p,f)$ above are all independent of $n$.
\end{lemma}

\section{Solvability and dependence on driving noises}\label{sec3}
\subsection{Runge--Kutta methods}
For $n\in\mathbb{N_+}$, denote the time step $h=\frac{T}{n}$ and $t_k=kh$, $k=0,\cdots,n$. We consider an $\mathbf{s}$-stage Runge--Kutta method of \eqref{sde}:
\begin{align}
Y^n_{t_{k+1},i}&=Y^n_{t_{k}}+\sum_{j=1}^{\mathbf{s}}a_{ij}V(Y^n_{t_{k+1},j})\Delta X_k,\label{middle}\\
Y^n_{t_{k+1}}&=Y^n_{t_{k}}+\sum_{i=1}^{\mathbf{s}}b_iV(Y^n_{t_{k+1},i})\Delta X_k,\label{scheme}
\end{align}
with $i,j=1,\cdots,\mathbf{s}$, $k=0,\cdots,n-1$, $\Delta X_k=X_{t_{k+1}}-X_{t_k}\in \mathbb{R}^d$ and $Y^n_{t_0}=y\in \mathbb{R}^m$. Here $Y^n_{t_{k+1},i}$, $i=1,\cdots,\mathbf{s}$, are called stage values of this Runge--Kutta method.

If the method is an implicit one, such as the midpoint scheme,
the solvability of \eqref{middle}-\eqref{scheme} should be taken into consideration. For SDEs driven by standard Brownian motions, the classical technique is to truncate each increment of Brownian motions to make increments become bounded, and give the solvability of implicit methods and convergence rates in mean square sense. However, this truncation technique is not suitable for fBms since their increments have nontrivial covariance. Based on Brouwer's theorem, Proposition \ref{sol} ensures the solvability of implicit Runge--Kutta methods in almost surely sense. 

\begin{proposition}\label{sol}
	If $V\in \mathcal{C}^{0}_b(\mathbb{R}^m;\mathbb{R}^{m\times d})$, then for arbitrary time step $h>0$, initial value $y$ and coefficients $\{a_{ij},b_{i}:i,j=1,\cdots,\mathbf{s}\}$, the $\mathbf{s}$-stage Runge--Kutta method \eqref{middle}-\eqref{scheme} has at least one solution for almost every $\omega$.
\end{proposition}
\begin{proof}
	Fix $h>0$ and $Y^n_{t_{k}}\in\mathbb{R}^{2m} $. 
	
	Let $Z_{1},\cdots,Z_{\mathbf{s}}\in\mathbb{R}^{2m}$ and $Z=(Z_{1}^\top,\cdots,Z_{\mathbf{s}}^\top)^\top\in\mathbb{R}^{2m\mathbf{s}}$. We define a map $\phi:\mathbb{R}^{2m\mathbf{s}}\rightarrow \mathbb{R}^{2m\mathbf{s}}$ with
	\begin{align*}
	\begin{split}
	&\phi(Z)=(\phi(Z)_{1}^\top,\cdots,\phi(Z)_{\mathbf{s}}^\top)^\top,\\ &\phi(Z)_{i}=Z_{i}-Y^n_{t_{k}}-\sum^{s}_{j=1}a_{ij}V(Z_{j})\Delta X_k(\omega),\ i=1,\cdots,\mathbf{s}.
	\end{split}
	\end{align*}
	It then suffices to prove that $\phi(Z)=0$ has at least one solution, which implies the solvability of \eqref{middle} and thus the solvability of the Runge--Kutta method.
	Let $c=\max\{|a_{ij}|:i,j=1,\cdots,\mathbf{s}\}$,  $\nu=\sup_{y\in\mathbb{R}^m}|V(y)|$ and 
	$$R=\sqrt{\mathbf{s}}|Y^n_{t_{k}}|+\mathbf{s}\sqrt{\mathbf{s}}c\nu|\Delta X_k(\omega)|+1,$$ we have that for any $|Z|= R$,
	\begin{align*}
	Z^\top\phi(Z)&=\sum^{\mathbf{s}}_{i=1}Z_{i}^\top\left(Z_{i}-Y^n_{t_{k}}-\sum^{\mathbf{s}}_{j=1}a_{ij}V(Z_{j})\Delta X_k(\omega)\right)\\
	&\geq|Z|\left(|Z|-\sqrt{\mathbf{s}}|Y^n_{t_{k}}|-\mathbf{s}\sqrt{\mathbf{s}}c\nu|\Delta X_k(\omega)|\right)>0.
	\end{align*}
	
	We aim to show that $\phi(Z)=0$ has a solution in the ball $B_{R}:=\{Z:|Z|\leq R\}$.
	Assume by contradiction that $\phi(Z)\neq 0$ for any $|Z|\leq R$. We define a continuous map $\psi$ by $\psi(Z)=-\frac{R\phi(Z)}{|\phi(Z)|}$. Since $\psi: B_{R}\rightarrow B_{R}$, $\psi$ has at least one fixed point $Z^{*}$ such that $Z^{*}=\psi(Z^{*})$ and $|Z^{*}|=R$. This leads to a contradiction since $|Z^{*}|^{2}=\psi(Z^{*})^\top Z^{*}=-\frac{R\phi(Z^{*})^\top Z^{*}}{|\phi(Z^{*})|}<0$. Therefore, $\phi$ has at least one solution.
\end{proof}

We construct the continuous version \eqref{middlec}-\eqref{schemec} for the Runge--Kutta method, taking advantages of the stage values $Y^n_{t_k,i}$. Indeed, the continuous version of $Y^n_{t_k,i}$ comes after that of $Y^n_{t_k}$. For $t\in(t_k,t_{k+1}]$, 
$\lceil t \rceil ^n:=t_{k+1}$. In particular, $t=t_k$ if and only if $t=\lceil t \rceil ^n$ for some $k=0,\cdots,n$. The continuous version reads
\begin{align}
Y^n_{t,i}&:=Y^n_{(t-h)\vee 0}
+\sum_{j=1}^{\mathbf{s}}\int_{(t-h)\vee 0}^{t}a_{ij}V(Y^n_{\lceil s \rceil ^n,j})dX_s,\quad  i=1,\cdots,\mathbf{s},\label{middlec}\\
Y^n_t&:=y+
\sum_{i=1}^{\mathbf{s}}\int_{0}^{t}b_iV(Y^n_{\lceil s \rceil ^n,i})dX_s,\label{schemec}
\end{align}
where $s\vee t$ denotes the maximum of $s$ and $t$.

To estimate the H\"older semi-norm of $Y^n_\cdot$ and $Y^n_{\cdot,i}$, we first introduce the discrete H\"older semi-norm for $f\in\mathcal{C}([0,T],\mathbb{R}^d)$:

\begin{align*}
\|f\|_{s,t,\beta,n}:=&\sup\bigg\{\frac{|f_v-f_u|}{(v-u)^\beta},~s\le u<v\le t,~u=\lceil u \rceil ^n,~v=\lceil v \rceil ^n\bigg\},\\
\|f\|_{\beta,n}:=&\sup\bigg\{\frac{|f_v-f_u|}{(v-u)^\beta},~0\le u<v\le T,~u=\lceil u \rceil ^n,v=\lceil v \rceil ^n\bigg\}.
\end{align*}

\begin{lemma}
	lf $V\in \mathcal{C}^{0}_b(\mathbb{R}^m;\mathbb{R}^{m\times d})$, then for any $n\in\mathbb{N_+}$ and $1/2<\beta<H$, $\|Y^{n}_\cdot\|_{\beta,n}$ and $\|Y^{n}_{\cdot,i}\|_{\beta,n}$, $i=1,\cdots,\mathbf{s}$, are all finite almost surely. More precisely, 
	\begin{align*}
	\|Y^{n}_\cdot\|_{\beta,n}&\le C(d,m,n,c,\nu,\mathbf{s})
	\|X\|_\beta<\infty,\quad a.s.,\\
	\|Y^{n}_{\cdot,i}\|_{\beta,n}&
	\le C(d,m,n,c,\nu,\mathbf{s})
	\|X\|_\beta<\infty,\quad a.s.,
	\end{align*}
	with $c=\max\{|a_{ij}|,|b_i|:i,j=1,\cdots,\mathbf{s}\}$ and $\nu=\sup_{y\in \mathbb{R}^m}|V(y)|$.
\end{lemma}

Inspired by \cite{HuEuler}, we give the following two lemmas as a priori estimates for the continuous version \eqref{middlec}-\eqref{schemec}. Based on them, Proposition \ref{varRK} shows that the stage values $Y^{n}_{\cdot,i}$ is coutinuously dependent with respect to the driving noises in H\"older semi-norm and so is $Y^{n}_{\cdot}$.

\begin{lemma}\label{fcalculus}
	Let $\alpha$, $\beta$ and $\beta'$ satisfy $\beta'>\alpha>1-\beta$. Then for any $s,t\in[0,T]$ such that $s<t$ and $s=\lceil s \rceil ^n$, there exists a constant $C=C(\alpha,\beta,\beta',T)$ such that 
	\begin{align*}
	\int_{s}^{t}(t-r)^{\alpha+\beta-1}\int_{s}^{r}\frac{(\lceil r \rceil ^n-\lceil u \rceil ^n)^{\beta'}}{(r-u)^{\alpha+1}}dudr\leq C (t-s)^{\beta+\beta'}.
	\end{align*}
\end{lemma}
\begin{proof}
	Suppose $T=1$ without loss of generality. By the definition of $\lceil \cdot \rceil ^n$, we have
	\begin{align*}
	&\int_{s}^{t}(t-r)^{\alpha+\beta-1}\int_{s}^{r}\frac{\big(\lceil r \rceil ^n-\lceil u \rceil ^n\big)^{\beta'}}{(r-u)^{\alpha+1}}dudr\\
	=&\int_{\lceil s \rceil ^n+1/n}^{t}(t-r)^{\alpha+\beta-1}\int_{\lceil s \rceil ^n}^{\lceil r \rceil ^n-\frac{1}{n}}\frac{\big(\lceil r \rceil ^n-\lceil u \rceil ^n\big)^{\beta'}}{(r-u)^{\alpha+1}}dudr\\
	=&\int_{\lceil s \rceil ^n+1/n}^{t}(t-r)^{\alpha+\beta-1}
	\bigg(\int_{\lceil s \rceil ^n}^{\lceil r \rceil ^n-\frac{2}{n}}
	+\int_{\lceil r \rceil ^n-\frac{2}{n}}^{\lceil r \rceil ^n-\frac{1}{n}}\bigg)
	\frac{\big(\lceil r \rceil ^n-\lceil u \rceil ^n\big)^{\beta'}}{(r-u)^{\alpha+1}}dudr\\
	=&:I_1+I_2.
	\end{align*}
	For the first term, since $r-u>\frac{1}{n}$ and $\lceil r \rceil ^n-\lceil u \rceil ^n<r-u+\frac1n$, we have
	\begin{align*}
	I_1&=\int_{\lceil s \rceil ^n+1/n}^{t}(t-r)^{\alpha+\beta-1}\int_{\lceil s \rceil ^n}^{\lceil r \rceil ^n-\frac{2}{n}}\frac{\big(\lceil r \rceil ^n-\lceil u \rceil ^n\big)^{\beta'}}{(r-u)^{\alpha+1}}dudr\\
	&\leq \int_{\lceil s \rceil ^n+1/n}^{t}(t-r)^{\alpha+\beta-1}\int_{\lceil s \rceil ^n}^{\lceil r \rceil ^n-\frac{2}{n}}\frac{2^{\beta'}(r-u)^{\beta'}}{(r-u)^{\alpha+1}}dudr\\
	&\leq C \int_{\lceil s \rceil ^n+1/n}^{t}(t-r)^{\alpha+\beta-1}(r-s)^{\beta'-\alpha}dr\\
	&\leq C\int_{\lceil s \rceil ^n+1/n}^{t}(t-r)^{\alpha+\beta-1}(t-s)^{\beta'-\alpha}dr\\
	&\leq C (t-s)^{\beta+\beta'}.
	\end{align*}
	For the second term, 
	\begin{align*}
	I_2&=\int_{\lceil s \rceil ^n+1/n}^{t}(t-r)^{\alpha+\beta-1}\int_{\lceil r \rceil ^n-\frac{2}{n}}^{\lceil r \rceil ^n-\frac{1}{n}}\frac{\big(\lceil r \rceil ^n-\lceil u \rceil ^n\big)^{\beta'}}{(r-u)^{\alpha+1}}dudr\\
	&\leq C (t-s)^{\alpha+\beta-1} \bigg(\frac{2}{n}\bigg)^{\beta'}\int_{\lceil s \rceil ^n+1/n}^{t}\bigg(r-\lceil r \rceil ^n+\frac{1}{n}\bigg)^{-\alpha}dr\\
	&\leq C (t-s)^{\alpha+\beta-1} \bigg(\frac{2}{n}\bigg)^{\beta'}
	\bigg(\frac{t-s}{\frac{1}{n}}\bigg)\bigg(\frac{1}{n}\bigg)^{-\alpha+1}\\
	&\leq C (t-s)^{\beta+\beta'}.
	\end{align*}
\end{proof}

\begin{lemma}\label{contro1}
	Let $\beta$ and $\beta'$ satisfy $\beta+\beta'>1$. Let $g\in \mathcal{C}^{1}_b(\mathbb{R}^m;\mathbb{R})$, $x\in \mathcal{C}([s,t];\mathbb{R})$ and $z\in \mathcal{C}([s,t];\mathbb{R}^m)$. If $\|x\|_\beta$ and $\|z\|_{\beta',n}$ are all finite for any $n\in\mathbb{N_+}$, then for any $s=\lceil s \rceil^n$ and $t=\lceil t \rceil^n$, 
	\begin{align*}
	\bigg|\int_{s}^{t}g(z_{\lceil r \rceil ^n})dx_r\bigg|\leq C(g,\beta,\beta',T)(1+\|z\|_{s,t,\beta',n}(t-s)^{\beta'}) \|x\|_\beta (t-s)^\beta.
	\end{align*}
\end{lemma}
\begin{proof}
	Considering the equivalence of norms in $\mathbb{R}^m$, we suppose $m=1$ here for simplicity without loss of generality.
	Let $\alpha$ satisfy $\alpha<\beta'$ and $\beta+\alpha>1$. According to the characterization of the integral in Section \ref{sec2}, 
	\begin{align*}
	\int_{s}^{t}g(z_{\lceil r \rceil ^n})dx_r=
	\int_{s}^{t}D^\alpha_{s+}g(z_{\lceil r \rceil ^n})D^{1-\alpha}_{t-}(x_r-x_t)dr.
	\end{align*}
	Combining the fractional Weyl derivatives, we have, for $s<r<t$,
	\begin{align*}
	\bigg|D^\alpha_{s+}g(z_{\lceil r \rceil ^n})\bigg|&\le C\Bigg(\Bigg|\frac{g(z_{\lceil r \rceil ^n})}{(r-s)^\alpha}\Bigg|
	+\int_{s}^{r}\frac{\big|g(z_{\lceil r \rceil ^n})-g(z_{\lceil u \rceil ^n})\big|}{(r-u)^{\alpha+1}}du\Bigg)\\
	&\le C\Bigg(\frac{1}{(r-s)^\alpha}
	+\int_{s}^{r}\frac{\big|z_{\lceil r \rceil ^n}-z_{\lceil u \rceil ^n}\big|}{(r-u)^{\alpha+1}}du\Bigg)\\
	&\le C\Bigg(\frac{1}{(r-s)^\alpha}
	+\|z\|_{s,t,\beta',n}\int_{s}^{r}\frac{\big(\lceil r \rceil ^n-\lceil u \rceil ^n\big)^{\beta'}}{(r-u)^{\alpha+1}}du\Bigg),
	\end{align*}
	and
	\begin{align*}
	\bigg|D^{1-\alpha}_{t-}(x_r-x_t)\bigg|&\le C\Bigg(\bigg|\frac{x_r-x_t}{(t-r)^{1-\alpha}}\bigg|
	+\int_{r}^{t}\frac{\big|x_r-x_u\big|}{(u-r)^{2-\alpha}}du\Bigg)\\
	& \le C \|x\|_\beta (t-r)^{\alpha+\beta-1}.
	\end{align*}
	Using Lemma \ref{fcalculus}, we obtian
	\begin{align*}
	&\bigg|\int_{s}^{t}g(z_{\lceil r \rceil ^n})dx_r\bigg|\\
	\le&
	\int_{s}^{t}\Big|D^\alpha_{s+}g(z_{\lceil r \rceil ^n})D^{1-\alpha}_{t-}(x_r-x_t)\Big|dr\\
	\le& C \int_{s}^{t}\Bigg(\frac{1}{(r-s)^\alpha}
	+\|z\|_{s,t,\beta',n}\int_{s}^{r}\frac{\big(\lceil r \rceil ^n-\lceil u \rceil ^n\big)^{\beta'}}{(r-u)^{\alpha+1}}du\Bigg)\|x\|_\beta (t-r)^{\alpha+\beta-1}dr\\
	\le& C(1+\|z\|_{s,t,\beta',n}(t-s)^{\beta'}) \|x\|_\beta (t-s)^\beta.
	\end{align*}
\end{proof}

\begin{proposition}\label{varRK}
	If $V\in \mathcal{C}^{1}_b(\mathbb{R}^m;\mathbb{R}^{m\times d})$ and $1/2<\beta<H$, then for any $n\in \mathbb{N_+}$,
	\begin{align*}
	\sum_{i=1}^{\mathbf{s}}\|Y^{n}_{\cdot,i}\|_{\beta}&\le C(c,\mathbf{s},V,\beta,T)\max\bigg\{\|X\|_{\beta},\|X\|^{1/\beta}_{\beta}\bigg\},\\
	\sum_{i=1}^{\mathbf{s}}\|Y^{n}_{\cdot,i}\|_{\infty}&\le \mathbf{s}|y|+C(c,\mathbf{s},V,\beta,T)\max\bigg\{\|X\|_{\beta},\|X\|^{1/\beta}_{\beta}\bigg\},\\
	\|Y^{n}_{\cdot}\|_{\beta}&\le C(c,\mathbf{s},V,\beta,T)\max\bigg\{\|X\|_{\beta},\|X\|^{1/\beta+1}_{\beta}\bigg\},\\
	\|Y^{n}_{\cdot}\|_{\infty}&\le |y|+C(c,\mathbf{s},V,\beta,T)\max\bigg\{\|X\|_{\beta},\|X\|^{1/\beta+1}_{\beta}\bigg\},
	\end{align*}
	where $c=\max\{|a_{ij}|,|b_i|:i,j=1,\cdots,\mathbf{s}\}$.
	
	Moreover, for some $C_0>0$ and $0\le s<t\le T$ such that $\|X\|_\beta|t-s|^\beta\le C_0$, the estimate can be improved to
	\begin{align*}
	\sum_{i=1}^{\mathbf{s}}\|Y^{n}_{\cdot,i}\|_{s,t,\beta}&\le C(c,\mathbf{s},V,\beta,T,C_0)\|X\|_{\beta},\\
	\|Y^{n}_{\cdot}\|_{s,t,\beta}&\le C(c,\mathbf{s},V,\beta,T,C_0)\|X\|_{\beta}.
	\end{align*}
\end{proposition}
\begin{proof}
	We first take $s=\lceil s \rceil^n$ and $t=\lceil t \rceil^n$, then Lemma \ref{contro1} yields
	\begin{align*}
	\big|Y^n_{t,i}-Y^n_{s,i}\big|
	\le &\sum_{j=1}^{\mathbf{s}}\bigg|\int_{0}^{(t-h)\vee 0}b_jV(Y^n_{\lceil r \rceil ^n,j})dX_r
	+\int_{(t-h)\vee 0}^{t}a_{ij}V(Y^n_{\lceil r \rceil ^n,j})dX_r\\
	&-\int_{0}^{(s-h)\vee 0}b_jV(Y^n_{\lceil r \rceil ^n,j})dX_r
	-\int_{(s-h)\vee 0}^{s}a_{ij}V(Y^n_{\lceil r \rceil ^n,j})dX_r\bigg|\\
	\le &\sum_{j=1}^{\mathbf{s}}\bigg|\int_{s}^{(t-h)\vee 0}b_jV(Y^n_{\lceil r \rceil ^n,j})dX_r\bigg|
	+\bigg|\int_{(t-h)\vee 0}^{t}a_{ij}V(Y^n_{\lceil r \rceil ^n,j})dX_r\bigg|\\
	&+\bigg|\int_{(s-h)\vee 0}^{s}b_jV(Y^n_{\lceil r \rceil ^n,j})dX_r\bigg|
	+\bigg|\int_{(s-h)\vee 0}^{s}a_{ij}V(Y^n_{\lceil r \rceil ^n,j})dX_r\bigg|\\
	\le & C(c,V,\beta,T)(1+\sum_{j=1}^{\mathbf{s}}\|Y^n_{\cdot,j}\|_{s,t,\beta,n}(t-s)^{\beta}) \|X\|_\beta (t-s)^\beta.
	\end{align*}
	Summing up above inequalities for all $i=1,\cdots,\mathbf{s}$ and dividing both sides by $(t-s)^\beta$, we have
	\begin{align*}
	\sum_{i=1}^{\mathbf{s}}\|Y^n_{\cdot,i}\|_{s,t,\beta,n}
	&\le  (C(c,\mathbf{s},V,\beta,T)\vee1)(1+\sum_{i=1}^{\mathbf{s}}\|Y^n_{\cdot,i}\|_{s,t,\beta,n}(t-s)^{\beta}) \|X\|_\beta\\
	&=:C_1(1+\sum_{i=1}^{\mathbf{s}}\|Y^n_{\cdot,i}\|_{s,t,\beta,n}(t-s)^{\beta}) \|X\|_\beta.
	\end{align*}
	
	If $n\ge 2T(2C_1 \|X\|_\beta)^{1/\beta}$, then there exist $N_0\in \mathbb{N_+}$ and $N_1=\frac{N_0T}{n}$ such that 
	$$(2C_1 \|X\|_\beta)^{-1/\beta}\le 2N_1 \le 2(2C_1 \|X\|_\beta)^{-1/\beta}.$$
	When $t-s=N_1$, considering the choice for $N_1$, we get 
	\begin{align*}
	\sum_{i=1}^{\mathbf{s}}\|Y^n_{\cdot,i}\|_{s,t,\beta,n}
	\le 2C_1\|X\|_\beta.
	\end{align*}
	When $t-s>N_1$, 
	\begin{align*}
	\sum_{i=1}^{\mathbf{s}}\|Y^n_{\cdot,i}\|_{s,t,\beta,n}
	&\le C \bigg(\bigg[\frac{t-s}{N_1}\bigg]+1\bigg)\sup_{r=\lceil r \rceil ^n\le t_{n-1}}\sum_{i=1}^{\mathbf{s}}\|Y^n_{\cdot,i}\|_{r,r+N_1,\beta,n}\frac{N_1^\beta}{(t-s)^\beta},
	\end{align*}
	where $[t]$ means the largest integer which is not larger than $t$.
	So we obtain  
	\begin{align*}
	\sum_{i=1}^{\mathbf{s}}\|Y^n_{\cdot,i}\|_{s,t,\beta,n}
	&\le C\bigg(\frac{T}{N_1}+1\bigg)^{1-\beta}C_1\|X\|_\beta\\
	&\le C\Big(\|X\|_\beta^{1/\beta-1}+1\Big)\|X\|_\beta\\
	&\le C\max\bigg\{\|X\|_{\beta},\|X\|^{1/\beta}_{\beta}\bigg\},
	\end{align*}
	where the second inequality is from the choice of $N_1$.
	If $n< 2T(2C_1 \|X\|_\beta)^{1/\beta}$, the definition of $Y^n_{\cdot,i}$ leads to
	\begin{align*}
	\sum_{i=1}^{\mathbf{s}}|Y^n_{t,i}-Y^n_{s,i}|
	\le C(t-s)\bigg(\frac{T}{n}\bigg) ^{-1+\beta}\|X\|_{\beta},
	\end{align*}
	and then
	\begin{align*}
	\sum_{i=1}^{\mathbf{s}}\|Y^n_{\cdot,i}\|_{s,t,\beta,n}
	\le C\|X\|_{\beta}^{1/\beta}.
	\end{align*}
	Now we can take any $0\le s<t\le T$ and suppose $t\in(t_k,t_{k+1}]$, then
	\begin{align*}
	\sum_{i=1}^{\mathbf{s}}\|Y^n_{\cdot,i}\|_{s,t,\beta}
	\le C\|X\|_{\beta}+\sum_{i=1}^{\mathbf{s}}\|Y^n_{\cdot,i}\|_{\lceil s \rceil ^n,t_k,\beta,n}
	\le C\max\bigg\{\|X\|_{\beta},\|X\|^{1/\beta}_{\beta}\bigg\}.
	\end{align*}
	Therefore, 
	\begin{align*}
	\|Y^{n}_{\cdot,i}\|_{\beta}&\le C(c,\mathbf{s},V,\beta,T)\max\bigg\{\|X\|_{\beta},\|X\|^{1/\beta}_{\beta}\bigg\},\\
	\|Y^{n}_{\cdot,i}\|_{\infty}&\le |y|+C(c,\mathbf{s},V,\beta,T)\max\bigg\{\|X\|_{\beta},\|X\|^{1/\beta}_{\beta}\bigg\}.
	\end{align*}
	
	Moreover, if $C_0 \in (0,1/C_1)$, then for any $0\le s<t\le T$ such that\\
	$\|X\|_\beta|t-s|^\beta\le C_0$, it can be improved to
	\begin{align*}
	\sum_{i=1}^{\mathbf{s}}\|Y^{n}_{\cdot,i}\|_{s,t,\beta}\le C(C_1,\beta,T,C_0)\|X\|_{\beta}.
	\end{align*}

	On the other hand, since
	\begin{align*}
	\big|Y^n_{t}-Y^n_{s}\big|
	\le& \sum_{i=1}^{\mathbf{s}}\bigg|\int_{s}^{t}b_iV(Y^n_{\lceil r \rceil ^n,j})dX_r\bigg|\\
	\le & C(1+\sum_{j=1}^{\mathbf{s}}\|Y^n_{\cdot,j}\|_{s,t,\beta,n}(t-s)^{\beta}) \|X\|_\beta (t-s)^\beta\\
	\le & C(c,\mathbf{s},V,\beta,T)\max\bigg\{\|X\|_{\beta},\|X\|^{1/\beta+1}_{\beta}\bigg\}(t-s)^\beta,
	\end{align*}
	we obtain that 
	\begin{align*}
	\|Y^{n}_{\cdot}\|_{\beta}&\le C(c,\mathbf{s},V,\beta,T)\max\bigg\{\|X\|_{\beta},\|X\|^{1/\beta+1}_{\beta}\bigg\},\\
	\|Y^{n}_{\cdot}\|_{\infty}&\le |y|+C(c,\mathbf{s},V,\beta,T)\max\bigg\{\|X\|_{\beta},\|X\|^{1/\beta+1}_{\beta}\bigg\},
	\end{align*}
	and for any $0\le s<t\le T$ such that  $\|X\|_\beta|t-s|^\beta\le C_0\in(0,1/C_1)$, 
	\begin{align*}
	\|Y^{n}_{\cdot}\|_{s,t,\beta}\le C(c,\mathbf{s},V,\beta,T,C_0)\|X\|_{\beta}.
	\end{align*}
	
\end{proof}

\subsection{Simplified  step-$N$ Euler schemes}
Fix an integer $N\ge 2$. For $n\in\mathbb{N_+}$, denote the time step $h=\frac{T}{n}$ and $t_k=kh$, $k=0,\cdots,n$. The simplified step-$N$ Euler scheme of \eqref{sde} is:
\begin{align}
Y^n_{t_{k+1}}&=Y^n_{t_{k}}+\sum_{w=1}^{N}\sum^d_{l_w,\cdots,l_1=1}\mathscr{V}_{l_w}\cdots\mathscr{V}_{l_1}I(Y^n_{t_k})\frac{\Delta X_k^{l_w}\cdots\Delta X_k^{l_1}}{N!},\label{NEuler}
\end{align}
where every $\mathscr{V}_{l}$ is identified with the first order differential operator $\sum_{q}V^q_l(y)\frac{\partial}{\partial y^q}$.
Note that $Y^n_{t_{k}}$ may stand for numerical solutions of different schemes in different subsections if there is no confusion. 
We consider its corresponding continuous version as before. For $t\in(t_k,t_{k+1}]$,
$\lfloor t \rfloor _n:=t_{k}$ and for $t=0$, 
$\lfloor t \rfloor _n:=0$. Then the continuous version of \eqref{NEuler} is 
\begin{align}
Y^n_{t}&=y+\int_{0}^{t}\sum_{w,l_w,\cdots,l_1}\mathscr{V}_{l_w}\cdots\mathscr{V}_{l_1}I(Y^n_{\lfloor s \rfloor _n})\frac{\Big(X_{\lceil s \rceil ^n}^{l_w}-X_{\lfloor s \rfloor _n}^{l_w}\Big)\cdots\Big(X_{\lceil s \rceil ^n}^{l_2}-X_{\lfloor s \rfloor _n}^{l_2}\Big)}{N!}dX^{l_1}_s.\label{cNEuler}
\end{align}

Similar to the analysis for Runge--Kutta methods, the continuous dependence of $Y^n_\cdot$ on the driving noises in H\"older semi-norm is obtained.

\begin{lemma}
	If $V\in \mathcal{C}^{N-1}_b(\mathbb{R}^m;\mathbb{R}^{m\times d})$, then for any $n\in\mathbb{N_+}$ and $1/2<\beta<H$, $\|Y^{n}_\cdot\|_{\beta,n}$ are all finite almost surely.
\end{lemma}

\begin{lemma}\label{fcalculus2}
	Let $\alpha$, $\beta$ and $\beta'$ satisfy $\beta'>\alpha>1-\beta$. Then for any $s,t\in[0,T]$ such that $s<t$ and $s=\lceil s \rceil ^n$, there exists a constant $C=C(\alpha,\beta,\beta',T)$ such that 
	\begin{align*}
	\int_{s}^{t}(t-r)^{\alpha+\beta-1}\int_{s}^{r}\frac{(\lfloor r \rfloor _n-\lfloor u \rfloor _n)^{\beta'}}{(r-u)^{\alpha+1}}dudr\leq K (t-s)^{\beta+\beta'}.
	\end{align*}
\end{lemma}
\begin{proof}
	Similar to the proof of Lemma \ref{fcalculus}.
\end{proof}

\begin{lemma}\label{contro2}
	Let $\beta$ and $\beta'$ satisfy $\beta+\beta'>1$. If $g\in \mathcal{C}^{1}_b(\mathbb{R}^m;\mathbb{R})$, $x^l\in \mathcal{C}([s,t];\mathbb{R})$, $l\in\{1,\cdots,d\}$, and $z\in \mathcal{C}([s,t];\mathbb{R}^m)$. If $\|x^l\|_\beta$, $\|z\|_{\beta',n}$ are all finite for any $l\in\{1,\cdots,d\}$, $n\in\mathbb{N_+}$, then for any $w\ge 2$, $s=\lceil s \rceil^n$ and $t=\lceil t \rceil^n$, 
	\begin{align*}
	&\bigg|\int_{s}^{t}g(z_{\lfloor r \rfloor _n})(x_{\lceil r \rceil _n}^{l_w}-x_{\lfloor r \rfloor _n}^{l_w})\cdots(x_{\lceil r \rceil ^n}^{l_2}-x_{\lfloor r \rfloor _n}^{l_2})dx^{l_1}_r\bigg|\\
	\le& C(g,\beta,\beta',T)(1+\|z\|_{s,t,\beta',n}(t-s)^{\beta'}) \|x^{l_w}\|_\beta\cdots\|x^{l_1}\|_\beta(t-s)^{\beta w},
	\end{align*}
	where $l_1,\cdots,l_w\in \{1,\cdots,d\}$.
\end{lemma}

\begin{proof}
	Let $\Phi(r)=g(z_{\lfloor r \rfloor _n})(x_{\lceil r \rceil ^n}^{l_w}-x_{\lfloor r \rfloor _n}^{l_w})\cdots(x_{\lceil r \rceil ^n}^{l_2}-x_{\lfloor r \rfloor _n}^{l_2})$. Taking $\alpha$ such that $\alpha<\beta'$ and $\beta+\alpha>1$, we first estimate the left fractional Weyl derivative of $\Phi$:
	\begin{align*}
	|D^\alpha_{s+}\Phi_r|\le \bigg|\frac{\Phi_r}{(r-s)^\alpha}\bigg|+\alpha\int_{s}^{r}\frac{\big|\Phi_r-\Phi_u\big|}{(r-u)^{\alpha+1}}du.
	\end{align*}
	For the first term, 
	\begin{align*}
	\bigg|\frac{\Phi_r}{(r-s)^\alpha}\bigg|\le C \|x^{l_w}\|_\beta\cdots\|x^{l_2}\|_\beta\bigg(\frac{T}{n}\bigg)^{\beta(w-1)}(r-s)^{-\alpha}.
	\end{align*}
	For the second term, we decompose $\Phi_r-\Phi_u$ into 
	{\small
		\begin{align*}
		&\Big[g(z_{\lfloor r \rfloor _n})\cdots(x_{\lceil r \rceil ^n}^{l_3}-x_{\lfloor r \rfloor _n}^{l_3})(x_{\lceil r \rceil ^n}^{l_2}-x_{\lfloor r \rfloor _n}^{l_2})-g(z_{\lfloor r \rfloor _n})\cdots(x_{\lceil r \rceil ^n}^{l_3}-x_{\lfloor r \rfloor _n}^{l_3})(x_{\lceil u \rceil ^n}^{l_2}-x_{\lfloor u \rfloor _n}^{l_2})\Big]\\
		+&\Big[g(z_{\lfloor r \rfloor _n})\cdots(x_{\lceil r \rceil ^n}^{l_3}-x_{\lfloor r \rfloor _n}^{l_3})(x_{\lceil u \rceil ^n}^{l_2}-x_{\lfloor u \rfloor _n}^{l_2})-g(z_{\lfloor r \rfloor _n})\cdots(x_{\lceil u \rceil ^n}^{l_3}-x_{\lfloor u \rfloor _n}^{l_3})(x_{\lceil u \rceil ^n}^{l_2}-x_{\lfloor u \rfloor _n}^{l_2})\Big]\\
		+&\cdots\\
		+&\Big[g(z_{\lfloor r \rfloor _n})(x_{\lceil u \rceil ^n}^{l_w}-x_{\lfloor u \rfloor _n}^{l_w})\cdots(x_{\lceil u \rceil ^n}^{l_2}-x_{\lfloor u \rfloor _n}^{l_2})-g(z_{\lfloor u \rfloor _n})(x_{\lceil u \rceil ^n}^{l_w}-x_{\lfloor u \rfloor _n}^{l_w})\cdots(x_{\lceil u \rceil ^n}^{l_2}-x_{\lfloor u \rfloor _n}^{l_2})\Big]\\
		=&:I_1+I_2\cdots+I_w.
		\end{align*}
	}
	We analyze each of them by
	\begin{align*}
	|I_1|&\le C\|x^{l_w}\|_\beta\cdots\|x^{l_2}\|_\beta\bigg(\frac{T}{n}\bigg)^{\beta(w-2)}\Big[(\lceil r \rceil ^n - \lceil u \rceil ^n)^\beta+(\lfloor r \rfloor _n - \lfloor u \rfloor _n)^\beta\Big],\\
	|I_2|&\le C\|x^{l_w}\|_\beta\cdots\|x^{l_2}\|_\beta\bigg(\frac{T}{n}\bigg)^{\beta(w-2)}\Big[(\lceil r \rceil ^n - \lceil u \rceil ^n)^\beta+(\lfloor r \rfloor _n - \lfloor u \rfloor _n)^\beta\Big],\\
	\cdots\\
	|I_w|&\le 
	C\|x^{l_w}\|_\beta\cdots\|x^{l_2}\|_\beta\bigg(\frac{T}{n}\bigg)^{\beta(w-1)}\|z\|_{s,t,\beta'}(\lfloor r \rfloor _n - \lfloor u \rfloor _n)^{\beta'}.
	\end{align*}
	Combining Lemma \ref{fcalculus2} and arguments in Lemma \ref{contro1}, we conclude the proof.
\end{proof}

\begin{proposition}\label{prop3}
	If $V\in \mathcal{C}^{N}_b(\mathbb{R}^m;\mathbb{R}^{m\times d})$ and $1/2<\beta<H$, then for any $n\in \mathbb{N_+}$,
	\begin{align*}
	\|Y^{n}\|_{\beta}&\le C(N,V,\beta,T)\max\bigg\{\|X\|_{\beta},\|X\|^{N-1+1/\beta}_{\beta}\bigg\},\\
	\|Y^{n}\|_{\infty}&\le |y|+C(N,V,\beta,T)\max\bigg\{\|X\|_{\beta},\|X\|^{N-1+1/\beta}_{\beta}\bigg\}.
	\end{align*}
	
	Moreover, for some $C_0>0$ and $0\le s<t\le T$ such that $\|X\|_\beta|t-s|^\beta\le C_0$, the estimate can be improved to
	\begin{align*}
	\|Y^{n}\|_{s,t,\beta}\le C(N,V,\beta,T,C_0)\|X\|_{\beta}.
	\end{align*}
\end{proposition}

\begin{proof}
	Take $s=\lceil s \rceil^n$ and $t=\lceil t \rceil^n$. Lemma \ref{contro2} yields
	\begin{align*}
	\big|Y^n_t-Y^n_s\big|
	\le &C(N,V,\beta,T)\sum_{w=1}^{N}\|X\|_\beta^w(t-s)^{\beta w}\big[1+\|Y^n\|_{s,t,\beta,n}(t-s)^\beta\big]\\
	=&:C_1 \sum_{w=1}^{N}\|X\|_\beta^w(t-s)^{\beta w}\big[1+\|Y^n\|_{s,t,\beta,n}(t-s)^\beta\big],
	\end{align*}
	since $\|X^{l_w}\|_\beta\cdots\|X^{l_1}\|_\beta \le \|X\|_\beta^w$. 
	Dividing both sides by $(t-s)^\beta$, we have
	\begin{align*}
	\|Y^n\|_{s,t,\beta,n}
	&\le C_1 \sum_{w=1}^{N}\|X\|_\beta^w(t-s)^{\beta (w-1)}\big[1+\|Y^n\|_{s,t,\beta,n}(t-s)^\beta\big].
	\end{align*}
	
	If $n\ge 2T(2NC_1 \|X\|_\beta)^{1/\beta}$, then there exist $N_0\in \mathbb{N_+}$ and $N_1=\frac{N_0T}{n}$ such that 
	$$(2NC_1 \|X\|_\beta)^{-1/\beta}\le 2N_1 \le 2(2NC_1 \|X\|_\beta)^{-1/\beta}.$$
	When $t-s=N_1$, considering the choice for $N_1$, we get 
	\begin{align*}
	C_1 \|X\|_\beta^w(t-s)^{\beta w}\le \frac{1}{2N},\quad \forall ~w=1,\cdots,N.
	\end{align*}
	So 
	\begin{align*}
	\|Y^n\|_{s,t,\beta,n} \le 2C_1\sum_{w=1}^{N}\|X\|_\beta^w(t-s)^{\beta (w-1)}.
	\end{align*}
	When $t-s>N_1$, 
	\begin{align*}
	\|Y^n\|_{s,t,\beta,n} 
	&\le C \bigg(\bigg[\frac{t-s}{N_1}\bigg]+1\bigg)\sup_{r=\lceil r \rceil ^n\le t_{n-1}}\|Y^n\|_{r,r+N_1,\beta,n}\frac{N_1^\beta}{(t-s)^\beta}.
	\end{align*}
	So we have that if $n\ge 2T(2NC_1 \|X\|_\beta)^{1/\beta}$,
	\begin{align*}
	\|Y^n\|_{s,t,\beta,n} 
	&\le C\max\bigg\{\|X\|_{\beta},\|X\|^{N-1+1/\beta}_{\beta}\bigg\}.
	\end{align*}
	If $n< 2T(2NC_1 \|X\|_\beta)^{1/\beta}$, the definition of $Y^n$ leads to
	\begin{align*}
	\|Y^n\|_{s,t,\beta,n}
	\le C\|X\|_{\beta}^{1/\beta}.
	\end{align*}
	
	Now we can take any $0\le s<t\le T$, 
	\begin{align*}
	\|Y^n\|_{s,t,\beta}
	\le C\|X\|^N_{\beta}+\|Y^n\|_{\lceil s \rceil ^n,\lfloor t \rfloor _n,\beta,n}
	\le C\max\bigg\{\|X\|_{\beta},\|X\|^{N-1+1/\beta}_{\beta}\bigg\}.
	\end{align*}
	Therefore, 
	\begin{align*}
	\|Y^{n}\|_{\beta}&\le C(c,\mathbf{s},V,\beta,T)\max\bigg\{\|X\|_{\beta},\|X\|^{N-1+1/\beta}_{\beta}\bigg\},\\
	\|Y^{n}\|_{\infty}&\le |y|+C(c,\mathbf{s},V,\beta,T)\max\bigg\{\|X\|_{\beta},\|X\|^{N-1+1/\beta}_{\beta}\bigg\}.
	\end{align*}
	
	Moreover, if $C_0 \in (0,(C_1N)^{-1})$, then for any $0\le s<t\le T$ such that  $\|X\|_\beta|t-s|^\beta\le C_0$, it can be improved to
	\begin{align*}
	\|Y^{n}\|_{s,t,\beta}\le C(N,V,\beta,T,C_0)\|X\|_{\beta}.
	\end{align*}

\end{proof}

\begin{remark}\label{Fernique}
	Note that Fernique's lemma implies that $\big\|\|X\|^N_{\beta}\big\|_{L^p(\Omega)}<\infty$ for any $p\ge 1$ and $N\in\mathbb{N_+}$ (see e.g. \cite[Remark 3.2]{HuEuler}).
\end{remark}

\section{Strong convergence rate}\label{sec4}
\subsection{Runge--Kutta methods}\label{sec4.1}

The order conditions on coefficients of Runge--Kutta methods are derived in this section to ensure the strong convergence rate is $2H-\frac{1}{2}$. For simplicity, we omit the range of indices in summations if it is not confusing.

\begin{align*}
Y_t-Y^n_t=&\bigg[\int_{0}^{t}V(Y_s)dX_s-\int_{0}^{t}V(Y^n_s)dX_s\bigg]\\
&+\bigg[\int_{0}^{t}V(Y^n_s)dX_s-\int_{0}^{t}\sum_{i}b_iV(Y^n_{\lceil s \rceil ^n,i})dX_s\bigg]\\
=&:L_t+R_t.
\end{align*}

For the first term, the Taylor expansion yields
\begin{align*}
L_t&=\int_{0}^{t}V(Y_s)dX_s-\int_{0}^{t}V(Y^n_s)dX_s\\
&=\sum_{l=1}^d\int_{0}^{t} \int_{0}^{1}
\nabla V_l(\theta Y_s+(1-\theta)Y^n_s) (Y_s-Y^{n}_s) d\theta dX^l_s.
\end{align*}

For the second term, fix any $t=\lceil t \rceil^n$, we have
\begin{align*}
R_t&=\int_{0}^{t}V(Y^n_s)dX_s-\int_{0}^{t}\sum_{i}b_iV(Y^n_{\lceil s \rceil ^n,i})dX_s\\
&=\sum^{nt/T-1}_{k=0}\int_{t_k}^{t_{k+1}}\bigg[V(Y^n_s)-\sum_{i}b_iV(Y^n_{t_{k+1},i})\bigg]dX_s.
\end{align*}
For any $i=1,\cdots,\mathbf{s}$, denote by $Y^{n,q}_t$ and $Y^{n,q}_{t,i}$ the $q$-th component of $Y^{n}_t$ and $Y^{n}_{t,i}$, respectively, $q=1,\cdots,m$.  We apply the Taylor expansion to $V(Y^n_{t_{k+1},i})$ at $Y^n_{t_{k}}$ or $Y^n_{t_{k+1}}$, then
\begin{align*}
V(Y^n_{t_{k+1},i})&=V(Y^n_{t_{k}})+\int_{0}^{1}\sum_{q}\partial_q V(\theta Y^n_{t_{k+1},i}+(1-\theta)Y^n_{t_{k}})(Y^{n,q}_{t_{k+1},i}-Y^{n,q}_{t_{k}}) d\theta\\
&=V(Y^n_{t_{k+1}})+\int_{0}^{1}\sum_{q}\partial_q V(\theta Y^n_{t_{k+1},i}+(1-\theta)Y^n_{t_{k+1}})(Y^{n,q}_{t_{k+1},i}-Y^{n,q}_{t_{k+1}}) d\theta,
\end{align*}
where $\partial_q$ denotes the partial differential operator with respect to the $q$-th variable. 
Assume $\eta\in [0,1]$, then for any $s\in(t_k,t_{k+1}]$
\begin{align*}
&V(Y^n_s)-\sum_{i}b_iV(Y^n_{t_{k+1},i})\\
=&\eta\bigg[ V(Y^n_s)-\sum_{i}b_iV(Y^n_{t_{k}})\bigg]\\
&-\eta \sum_{i,q}b_i\int_{0}^{1}\partial_q V(\theta Y^n_{t_{k+1},i}+(1-\theta)Y^n_{t_{k}})(Y^{n,q}_{t_{k+1},i}-Y^{n,q}_{t_{k}}) d\theta\\
&+(1-\eta)\bigg[ V(Y^n_s)-\sum_{i}b_iV(Y^n_{t_{k+1}})\bigg]\\
&-(1-\eta)\sum_{i,q}b_i\int_{0}^{1}\partial_q V(\theta Y^n_{t_{k+1},i}+(1-\theta)Y^n_{t_{k+1}})(Y^{n,q}_{t_{k+1},i}-Y^{n,q}_{t_{k+1}}) d\theta\\
=&:\eta R^1_s+\eta R^2_s+(1-\eta)R^3_s+(1-\eta)R^4_s.
\end{align*}
Since \eqref{schemec} implies
\begin{align*}
R^1_s&=V(Y^n_s)-\sum_{i}b_iV(Y^n_{t_{k}})\\ 
&=\bigg[ V(Y^n_{t_k}) +\int_{0}^1\sum_{q}\partial_q V(\theta Y^n_s+(1-\theta)Y^n_{t_k})(Y^{n,q}_s-Y^{n,q}_{t_k})d\theta  \bigg]
-\sum_{i}b_iV(Y^n_{t_{k}})\\
&=\bigg[ V(Y^n_{t_k}) +\sum_{q,i}\partial_q V(Y^n_{t_k})\int_{t_k}^{s}b_iV^q(Y^n_{t_{k+1},i})dX_u+E^1_s \bigg]
-\sum_{i}b_iV(Y^n_{t_{k}}),\\
\end{align*}
we propose the first condition that $\sum_{i=1}^{\mathbf{s}}b_i=1$. Then
\begin{align}\label{R1}
\int_{t_k}^{t_{k+1}}R^1_sdX_s=\int_{t_k}^{t_{k+1}}\sum_{q,i} \partial_q V(Y^n_{t_k})\int_{t_k}^{s}b_iV^q(Y^n_{t_{k+1},i})dX_udX_s
+\int_{t_k}^{t_{k+1}}E^1_sdX_s
\end{align}
with $E^1$ denoting the remainder term of $R^1$, so as $E^2$, $E^3$, $E^4$ in the following analysis.
Similarly, 
\begin{align}\label{R3}
\int_{t_k}^{t_{k+1}}R^3_sdX_s=-\int_{t_k}^{t_{k+1}}\sum_{q,i} \partial_q V(Y^n_{t_{k+1}})\int_{s}^{t_{k+1}}b_iV^q(Y^n_{t_{k+1},i})dX_udX_s
+\int_{t_k}^{t_{k+1}}E^3_sdX_s.
\end{align}
For $R^2_s$ and $R^4_s$, it follows from \eqref{middle} and \eqref{scheme} that  
\begin{align*}
Y^n_{t_{k+1},i}-Y^n_{t_{k}}&=\sum_{j}a_{ij}V(Y^n_{t_{k+1},j})\Delta X_k,\\
Y^n_{t_{k+1},i}-Y^n_{t_{k+1}}&=-\sum_{j}b_jV(Y^n_{t_{k+1},j})\Delta X_k
+\sum_{j}a_{ij}V(Y^n_{t_{k+1},j})\Delta X_k.\\
\end{align*}
Then,
\begin{align*}
\int_{t_k}^{t_{k+1}}R^2_sdX_s=&-\int_{t_k}^{t_{k+1}}\sum_{i,q,j}b_i \partial_q V(Y^n_{t_{k+1},i})\bigg[\int_{t_k}^{t_{k+1}}a_{ij}V^q(Y^n_{t_{k+1},j})dX_u\bigg]dX_s\\
&+\int_{t_k}^{t_{k+1}}E^2_sdX_s,\\
\int_{t_k}^{t_{k+1}}R^4_sdX_s=&-\int_{t_k}^{t_{k+1}}\sum_{i,q,j}b_i \partial_q V(Y^n_{t_{k+1},i})\bigg[\int_{t_k}^{t_{k+1}}(a_{ij}-b_j)V^q(Y^n_{t_{k+1},j})dX_u\bigg]dX_s\\
&+\int_{t_k}^{t_{k+1}}E^4_sdX_s.
\end{align*}

Taking the Taylor expansion to both $V(Y^n_{t_{k+1},i})$ and $V(Y^n_{t_{k+1},j})$ at $Y^n_{t_{k}}$ in above two expressions and choosing $\eta=\frac{1}{2}$, we propose another condition $\sum^{\mathbf{s}}_{i=1}b_i(\sum_{j=1}^{\mathbf{s}}b_j-2a_{ij})=0$ such that terms contain $2$nd-level iterated integrals of $X$ in the form of the L\'evy area type processes \eqref{multiple} vanish. Therefore, the leading term of $R_t$ only appears in \eqref{R1} and \eqref{R3}, which contains $2$nd-level iterated integrals of $X$ in the form of \eqref{levyB}-\eqref{levyt}. Using the Tayor expansion again to $V(Y^n_{t_k})$ and $V(Y^n_{t_{k+1}})$ at $Y^n_{t_{k+1},i}$ in \eqref{R1} and \eqref{R3}, we obtain the leading term of $R_t$ is 
\begin{align}\label{Rtleading}
R^{lead}_t=\frac{1}{2}\sum_{k=0}^{nt/T-1}\sum_{i,q,l,l'}b_i&\bigg[\int_{t_k}^{t_{k+1}}\int_{t_k}^{s} \partial_q V_l(Y^n_{t_{k+1},i})V^q_{l'}(Y^n_{t_{k+1},i})dX^{l'}_udX^l_s\\
&-\int_{t_k}^{t_{k+1}}\int_{s}^{t_{k+1}}  \partial_q V_l(Y^n_{t_{k+1},i})V^q_{l'}(Y^n_{t_{k+1},i})dX^{l'}_udX^l_s\bigg].
\end{align}
The remainder $R_t-R^{lead}_t$
contains $3$rd-level iterated integrals of $X$ in each interval $(t_k,t_{k+1}]$ in \eqref{multiple}.

\begin{theorem}\label{theorem1}
	Suppose $V\in \mathcal{C}^{3}_b(\mathbb{R}^m;\mathbb{R}^{m\times d})$ and $H>1/2$. Denote $c_i=\sum^{\mathbf{s}}_{j=1}a_{ij}$. If it holds that
	\begin{align}\label{condition}
	\sum^{\mathbf{s}}_{i=1}b_i=1\quad {\rm and}\quad  \sum^{\mathbf{s}}_{i=1}b_ic_{i}=1/2,
	\end{align}
	then the strong convergence rate of Runge--Kutta method for \eqref{sde} is $2H-\frac{1}{2}$. More precisely, there exists a constant $C$ independent of $n$ such that
	\begin{align}\label{strong}
	\bigg\|\sup_{t\in[0,T]}|Y_t-Y^n_t|\bigg\|_{L^p(\Omega)}\le C h^{2H-\frac{1}{2}},\quad p\ge 1,
	\end{align}
	where $h=\frac{T}{n}$ and $Y^n_t$ is defined by \eqref{schemec}.
\end{theorem}
\begin{proof}
	Notice that condition \eqref{condition} ensures the expression of $R^{lead}$ in the form of \eqref{Rtleading}. Then Lemma \ref{multi} and Lemma \ref{trans} combined with Proposition \ref{varRK} lead to
	\begin{align}
	\big\|R^{lead}_t-R^{lead}_s\big\|_{L^p(\Omega)}\le C(t-s)^{\frac{1}{2}}h^{2H-\frac{1}{2}},\quad p\ge 1,~t=\lceil t \rceil^n,~s=\lceil s \rceil^n. \label{orderlead}
	\end{align}
	Similarly, based on \eqref{multiple}, we have 
	\begin{align*}
	\big\|(R_t-R^{lead}_t)-(R_s-R^{lead}_s)\big\|_{L^p(\Omega)}\le C(t-s)^{\frac{1}{2}}h^{2H},\quad p\ge 1,~t=\lceil t \rceil^n,~s=\lceil s \rceil^n.
	\end{align*}
	
	Next, for the estimate of $L_t$, recall that
	\begin{align*}
	L_t=\sum_{l}\int_{0}^{t} \int_{0}^{1}
	\nabla V_l(\theta Y_s+(1-\theta)Y^n_s) (Y_s-Y^{n}_s) d\theta dX^l_s=:\sum_{l}\int_{0}^{t} S^{l}_s (Y_s-Y^{n}_s) dX^l_s.
	\end{align*}
	We introduce two linear equations defined through $S^{l}_s$. Let matrices $\Lambda^n$ and $\Gamma^n$ satisfy the linear equations:
	\begin{align*}
	\Lambda^{n}_t&=I+\sum_{l}\int_{0}^{t}S^{l}_s\Lambda^{n}_{s}dX_s^l,\\
	\Gamma^{n}_{t}&=I-\sum_{l}\int_{0}^{t}\Gamma^{n}_{s}S^{l}_sdX_s^l,
	\end{align*}
	where $I\in \mathbb{R}^{m\times m}$ denotes the identity matrix. Using the chain rule, we know that $\Lambda^n \Gamma^n=I$.
	Applying Proposition \ref{varRK}, Remark \ref{Fernique} and \cite[Lemma 3.1 ($ii$)]{HuEuler}, we have 
	\begin{align}
	\max\bigg\{\|\|\Lambda^n\|_{\infty}\|_{L^p(\Omega)}, \|\|\Lambda^n\|_{\beta}\|_{L^p(\Omega)},
	\|\|\Gamma^n\|_{\infty}\|_{L^p(\Omega)}, \|\|\Gamma^n\|_{\beta}\|_{L^p(\Omega)}\bigg\}\le C,\quad p\ge 1.\label{lambda}
	\end{align}
	
	It can be verified that $Y_t-Y_t^n=\Lambda^n_t\int_{0}^{t}\Gamma^n_sdR_s$, 
	so the H\"older inequality impies that 
	\begin{align*}
	\|\|Y-Y^n\|_\infty  \|_{L^p(\Omega)}&\le 
	\bigg\|\|\Lambda^n\|_\infty \bigg\|\int_{0}^{\cdot}\Gamma^n_sdR_s\bigg\|_\infty \bigg\|_{L^p(\Omega)}\\
	&\le \|\|\Lambda^n\|_\infty \|_{L^{2p}(\Omega)}\bigg\|\bigg\|\int_{0}^{\cdot}\Gamma^n_sdR_s\bigg\|_\infty \bigg\|_{L^{2p}(\Omega)}
	,\quad p\ge 1.
	\end{align*}
	Let $f^n_t=n^{2H-\frac{1}{2}} \int_{0}^{t}\Gamma^n_sdR_s  $. It suffices to prove that $\|\|f^n\|_\infty \|_{L^{q}(\Omega)}\le C$, for any $q\ge 1$.
	If there exists some $k\in\{1,\cdots,n\}$ such that $0\le s<t_k<t\le T$, we decomposite  $\int_{s}^{t}\Gamma^n_udR_u  $ by 
	\begin{align*}
	\int_{s}^{t}\Gamma^n_udR_u&=\int^{\lceil s \rceil ^n}_{s}\Gamma^n_udR_u+\int_{\lceil s \rceil ^n}^{\lfloor t \rfloor _n}\Gamma^n_udR_u+\int_{\lfloor t \rfloor _n}^t\Gamma^n_udR_u\\
	&=\int^{\lceil s \rceil ^n}_{s}\Gamma^n_udR_u
	+\int_{\lceil s \rceil ^n}^{\lfloor t \rfloor _n}\Gamma^n_{\lfloor u \rfloor _n}dR_u
	+\int_{\lceil s \rceil ^n}^{\lfloor t \rfloor _n}  \int_{\lfloor u\rfloor _n}^{u}d\Gamma^n_v dR_u
	+\int_{\lfloor t \rfloor _n}^t\Gamma^n_udR_u.
	\end{align*}
	For the first term, combining the definitions of $Y^n_t$ and $R_t$, we have
	\begin{align*}
	\int^{\lceil s \rceil ^n}_{s}\Gamma^n_udR_u
	= \int^{\lceil s \rceil ^n}_{s} \Gamma^n_u 
	\Big[ V(Y^n_u)-\sum_{i} b_i V(Y^n_{\lceil u \rceil ^n,i})\Big]dX_u.
	\end{align*}
	By the Taylor expansion and the property of Young's integral (see e.g. \cite{Lyons}), for any $\frac12<\beta<H$,
	\begin{align*}
	\bigg|\int^{\lceil s \rceil ^n}_{s}\Gamma^n_udR_u\bigg|&\le 
	C(\beta,V,T)\|X\|_\beta^2(\|\Gamma^n\|_\infty+\|\Gamma^n\|_\beta)|\lceil s \rceil ^n-s|^{\beta}h^{\beta}\\
	&\le 
	C(\beta,V,T)\|X\|_\beta^2(\|\Gamma^n\|_\infty+\|\Gamma^n\|_\beta)|t-s|^{\frac12-2(H-\beta)}n^{-(2H-\frac12)}.
	\end{align*}
	For the second term, according to \eqref{orderlead}-\eqref{lambda} and Lemma \ref{trans}, we have 
	\begin{align*}
	\bigg\|n^{2H-\frac12}\int_{\lceil s \rceil ^n}^{\lfloor t \rfloor _n}\Gamma^n_{\lfloor u \rfloor _n}dR_u\bigg\|_{L^q(\Omega)}\le C |t-s|^{\frac12}.
	\end{align*}
	For the third term, combining the definitions of $\Gamma^n_t$ and $R_t$, we know that it contains the $3$rd-level iterated integrals of $X$, then
	\begin{align*}
	\bigg\|n^{2H}\int_{\lceil s \rceil ^n}^{\lfloor t \rfloor _n}  \int_{\lfloor u\rfloor _n}^{u}d\Gamma^n_v dR_u\bigg\|_{L^q(\Omega)}\le C |t-s|^{\frac12}.
	\end{align*}
	For the fourth term, using similar arguments as the first one, we have
	\begin{align*}
	\bigg|\int_{\lfloor t \rfloor _n}^t\Gamma^n_udR_u\bigg|\le 
	C(\beta,V,T)\|X\|_\beta^2(\|\Gamma^n\|_\infty+\|\Gamma^n\|_\beta)|t-s|^{\frac12-2(H-\beta)}n^{-(2H-\frac12)}.
	\end{align*}
	If $t_k\le s<t\le t_{k+1}$, it holds that
	\begin{align*}
	\bigg|\int_{s}^t\Gamma^n_udR_u\bigg|\le 
	C(\beta,V,T)\|X\|_\beta^2(\|\Gamma^n\|_\infty+\|\Gamma^n\|_\beta)|t-s|^{\frac12-2(H-\beta)}n^{-(2H-\frac12)}.
	\end{align*}
	Therefore, for any $0\le s<t \le T$ and $q\ge 1$, we obtain 
	\begin{align*}
	\|f^n_t-f^n_s\|_{L^q(\Omega)}\le C\big(|t-s|^{\frac{1}{2}}+|t-s|^{\frac12-2(H-\beta)}\big).
	\end{align*}
	
	If $q>4$, we take $\beta$ such that $\max\{\frac{1}{2},H-\frac{1}{2q}\}<\beta<H$ and take $\alpha=\frac{1}{2}-\frac{1}{q}$ such that $\alpha\in(\frac{1}{q},\frac12)$, then Lemma \ref{Besov} yields that 
	\begin{align*}
	\mathbb{E}\Big[\|f^n\|^q_\infty\Big]
	&\le T^{\alpha q-1}\mathbb{E}\Big[\|f^n\|^q_{\alpha-\frac{1}{q}}\Big]\\
	&\le C \int_{0}^{T} \int_{0}^{T} \frac{\mathbb{E}\big[|f^n_t-f^n_s|^q\big]}{|t-s|^{1+q\alpha}}dsdt\\
	&\le C \int_{0}^{T} \int_{0}^{T} \frac{|t-s|^{\frac{q}{2}}+|t-s|^{\frac{q}{2}-2q(H-\beta)}}{|t-s|^{1+q\alpha}}dsdt\\
	&\le C.
	\end{align*}
\end{proof}

\begin{remark}\label{com}
	If the noise is one-dimensional or the diffusion term satisfies the following commutative condition
	\begin{align*}
	\sum_{q}\partial_q V_l V^q_{l'}=\sum_{q}\partial_q V_{l'} V^q_{l},\quad 1<l<l'\le d,
	\end{align*}
	then Fubini's theorem shows
	\begin{align*}
	R^{lead}_t=\frac{1}{2}\sum_{k=0}^{nt/T-1}\sum_{i,q}\sum_{l\neq 1}b_i\bigg[\partial_q V_l(Y^n_{t_{k+1},i})V^q_{1}(Y^n_{t_{k+1},i})&-\partial_q V_1(Y^n_{t_{k+1},i})V^q_{l}(Y^n_{t_{k+1},i})\bigg]\\
	\bigg[\int_{t_k}^{t_{k+1}}\int_{t_k}^{s} dudX^l_s
	&-\int_{t_k}^{t_{k+1}}\int_{s}^{t_{k+1}}  dudX^l_s\bigg].
	\end{align*}
	As a result, the strong convergence rate in \eqref{strong} is $H+\frac12$ from \eqref{levyt}.
\end{remark}

\begin{remark}
	As $H$ goes to $\frac{1}{2}$, the rate tends to $\frac{1}{2}$ in Theorem \ref{theorem1} and to $1$ in Remark \ref{com}, respectively. This is consistent with classical results for SDEs driven by standard Brownian motions in Stratonovich sense.
\end{remark}

\begin{remark}
	If $d=1$ or both drift and diffusion terms satisfy the commutative condition
	\begin{align*}
	\sum_{q}\partial_q V_l V^q_{l'}=\sum_{q}\partial_q V_{l'} V^q_{l},\quad 1\le l< l'\le d,
	\end{align*}
	then $R^{lead}_t=0$ and the strong convergence rate is $2H$ from \eqref{multiple}. In other words, we recover the conditions for order $2$ for deterministic ODEs if we take $H=1$ formally when $X_t=t$.
\end{remark}

\begin{remark}
	For more general case, if the Hurst parameter of $X^i$ is $H_i$, $i=1,...,d$, satisfying $H_1\ge\cdots\ge H_d>\frac12$, Lemma \ref{multi} and Theorem \ref{theorem1} with some revisions imply that the strong convergence rate is $H_{d-1}+H_d-\frac{1}{2}$.
\end{remark}

\subsection{Simplified step-$N$ Euler schemes}\label{theo2}
For simplicity, we take $N=2$ in the following. Indeed, our approach gives the same strong convergence rate $2H-\frac12$ of simplified step-$N$ Euler schemes for $N\ge 2$.

Recall that the simplified step-$2$ Euler scheme is 
\begin{align}\label{step2}
Y^n_{t_{k+1}}&=Y^n_{t_{k}}+\sum_{l}V(Y^n_{t_k})\Delta X_k^{l}+\frac{1}{2}\sum_{l,l',q}\partial_q V_l(Y^n_{t_k})V^q_{l'}(Y^n_{t_k})\Delta X_k^{l'}\Delta X_k^{l}.
\end{align}
The corresponding continuous version is 
\begin{align}\label{cstep2}
Y^n_{t}&=y+\int_{0}^{t}V(Y^n_{\lfloor s \rfloor_n})dX_s+\frac{1}{2}\int_{0}^{t}\sum_{l,l',q}\partial_q V_l(Y^n_{\lfloor s \rfloor_n})V^q_{l'}(Y^n_{\lfloor s \rfloor_n})\Big(X^{l'}_{\lceil s \rceil ^n}-X_{\lfloor s \rfloor_n}^{l'}\Big)dX_s^{l}.
\end{align}

To gain a sharp convergence order of the simplified step-$2$ Euler scheme, we compare it with the following $2$-stage Runge--Kutta method (the Heun's method)
\begin{align}
Z^n_{t_{k+1},1}&= Z^n_{t_{k}},\label{Z1}\\
Z^n_{t_{k+1},2}&= Z^n_{t_{k}}+V(Z^n_{t_{k+1},1})\Delta X_k,\label{Z2}\\
Z^n_{t_{k+1}}&=Z^n_{t_{k}}+\frac{1}{2}V(Z^n_{t_{k+1},1})\Delta X_k+\frac{1}{2}V(Z^n_{t_{k+1},2})\Delta X_k,\label{Z3}
\end{align}
which satisfies condition \eqref{condition}. We introduce two similar stage values for 
$Y^n_{t_{k+1}}$:
\begin{align*}
Y^n_{t_{k+1},1}&= Y^n_{t_{k}},\\
Y^n_{t_{k+1},2}&= Y^n_{t_{k}}+V(Y^n_{t_{k+1},1})\Delta X_k.\\
\end{align*}
Notice that 
\begin{align*}
V_l(Y^n_{t_{k+1},2})=V_l(Y^n_{t_{k}})+\bigg[\int_{0}^{1}\sum_{q,l'}\partial_q V_l(\theta Y^n_{t_{k+1},2}+(1-\theta)Y^n_{t_{k}})d\theta\bigg]V^q_{l'}(Y^n_{t_k})\Delta X^{l'}_k
\end{align*}
and
\begin{align*}
&\partial_q V_l(\theta Y^n_{t_{k+1},2}+(1-\theta)Y^n_{t_{k}})\\
=&\partial_q V_l(Y^n_{t_k})\\
&+\sum_{q',l''}
\bigg[\int_{0}^{1}\partial_{q'}\partial_q V_l ( \theta' \big[\theta Y^n_{t_{k+1},2} +(1-\theta)Y^n_{t_{k}} \big] +(1-\theta')  Y^n_{t_{k}} )d\theta'\bigg]
(\theta V_{l''}^{q'}(Y_{t_k}^n)\Delta X_k^{l''})\\
=&:\partial_q V_l(Y^n_{t_k})+G^n_{q,l,k}(\theta).
\end{align*}
Therefore, 
\begin{align*}
&V_l(Y^n_{t_{k+1},2})\\
=&V_l(Y^n_{t_{k}})+\sum_{q,l'}\bigg[\partial_q V_l(Y^n_{t_k}) + \int_{0}^{1} G^n_{q,l,k}(\theta)d\theta \bigg] V^q_{l'}(Y^n_{t_k})\Delta X^{l'}_k\\
=&V_l(Y^n_{t_{k}})+\sum_{q,l'}\partial_q V_l(Y^n_{t_k}) V^q_{l'}(Y^n_{t_k})\Delta X^{l'}_k+ \sum_{q,l'}\bigg[\int_{0}^{1} G^n_{q,l,k}(\theta)d\theta \bigg] V^q_{l'}(Y^n_{t_k})\Delta X^{l'}_k\\
=&:V_l(Y^n_{t_{k}})+\sum_{q,l'}\partial_q V_l(Y^n_{t_k}) V^q_{l'}(Y^n_{t_k})\Delta X^{l'}_k- G^n_{l,t_{k+1}}.
\end{align*}

Scheme \eqref{step2} and its continuous version \eqref{cstep2} can be transformed as
\begin{align}
Y^n_{t_{k+1}}&=Y^n_{t_{k}}+\frac{1}{2}V(Y^n_{t_{k+1},1})\Delta X_k+\frac{1}{2}V(Y^n_{t_{k+1},2})\Delta X_k+\frac{1}{2}\sum_{l}G^n_{l,t_{k+1}}\Delta X_k^l,\label{step2RK2}\\
Y^n_t&=y+
\frac{1}{2}\int_{0}^{t}V(Y^n_{\lceil s \rceil ^n,1})dX_s^l+
\frac{1}{2}\int_{0}^{t}V(Y^n_{\lceil s \rceil ^n,2})dX_s^l+
\frac{1}{2}\sum_{l}\int_{0}^{t}G^n_{l,\lceil s \rceil ^n}dX_s^l
,\label{cstep2RK2}
\end{align}
where $\frac{1}{2}\sum_{l}\int_{0}^{t}G^n_{l,\lceil s \rceil ^n}dX_s^l$ contains 
$3$rd-level iterated integrals of $X$. We define the continuous versions for the stage values $Y^n_{t_{k+1},1}$ and $Y^n_{t_{k+1},2}$:

\begin{align*}
Y^n_{t,1}&= Y^n_{(t-h)\vee 0},\\
Y^n_{t,2}&= Y^n_{(t-h)\vee 0}+\int_{(t-h)\vee 0}^{t}V(Y^n_{\lfloor s \rfloor_n,1})d X_s.\\
\end{align*}
The continuous dependence of $Y^n_{\cdot,1}$ and $Y^n_{\cdot,2}$ on the driving noises follows from Proposition \ref{prop3}.
\begin{proposition}\label{lemmastep2}
	If $V\in \mathcal{C}^{N}_b(\mathbb{R}^m;\mathbb{R}^{m\times d})$ and $1/2<\beta<H$, then for any $n\in \mathbb{N_+}$,
	\begin{align*}
	\|Y^n_{\cdot,2}\|_{\beta}&\le C(N,V,\beta,T)\max\bigg\{\|X\|_{\beta},\|X\|^{N-1+1/\beta}_{\beta}\bigg\},\\
	\|Y^n_{\cdot,2}\|_{\infty}&\le |y|+C(N,V,\beta,T)\max\bigg\{\|X\|_{\beta},\|X\|^{N-1+1/\beta}_{\beta}\bigg\}.
	\end{align*}
	
	Moreover, for some $C_0>0$ and $0\le s<t\le T$ such that $\|X\|_\beta|t-s|^\beta\le C_0$, the estimate can be improved to
	\begin{align*}
	\|Y^n_{\cdot,2}\|_{s,t,\beta}\le C(N,V,\beta,T,C_0)\|X\|_{\beta}.
	\end{align*}
\end{proposition}

Based on \eqref{step2RK2}-\eqref{cstep2RK2} and arguments in subsection \ref{sec4.1}, we obtain that the simplified step-$2$ Euler scheme has the same leading term as the one of scheme \eqref{Z1}-\eqref{Z3}:
\begin{align*}
R^{lead}_t=\frac{1}{4}\sum_{k=0}^{nt/T-1}\sum_{i,q,l,l'}&\bigg[\int_{t_k}^{t_{k+1}}\int_{t_k}^{s} \partial_q V_l(Y^n_{t_{k+1},i})V^q_{l'}(Y^n_{t_{k+1},i})dX^{l'}_udX^l_s\\
&-\int_{t_k}^{t_{k+1}}\int_{s}^{t_{k+1}}  \partial_q V_l(Y^n_{t_{k+1},i})V^q_{l'}(Y^n_{t_{k+1},i})dX^{l'}_udX^l_s\bigg],\quad t=\lceil t \rceil^n.
\end{align*}

Therefore, we get the same strong convergence rate for the simplified step-$2$ Euler scheme as in Theorem \ref{theorem1}.

\begin{theorem}\label{theorem2}
	If $N\ge 2$, $V\in \mathcal{C}^{N+1}_b(\mathbb{R}^m;\mathbb{R}^{m\times d})$ and $H>1/2$, then the simplified step-$N$ Euler scheme for \eqref{sde} is $2H-\frac{1}{2}$. More precisely, there exists a constant $C$ independent of $n$ such that
	\begin{align*}
	\bigg\|\sup_{t\in[0,T]}|Y_t-Y^n_t|\bigg\|_{L^p(\Omega)}\le C h^{2H-\frac{1}{2}},\quad p\ge 1,
	\end{align*}
	where $h=\frac{T}{n}$ and $Y^n_t$ is defined by \eqref{cNEuler}.
\end{theorem}

In particular, we have 
\begin{align*}
\bigg\|\max_{1\le k\le n}|Y_{t_k}-Y^n_{t_k}|\bigg\|_{L^p(\Omega)}\le C h^{2H-\frac{1}{2}},\quad p\ge 1.
\end{align*}
If we consider linear interpolation of $Y_{t_k}^n$, i.e., for any $t\in(t_k,t_{k+1}]$,
\begin{align}
Y_{t}^{n,linear}:=Y^n_{t_k}+\frac{t-t_k}{h}(Y^n_{t_{k+1}}-Y^n_{t_k}),\label{Milsteinlinear}
\end{align}
then
\begin{align*}
&\bigg\|\sup_{t\in[0,T]}|Y_t-Y^{n,linear}_t|\bigg\|_{L^p(\Omega)}\\
\le& \bigg\|\sup_{t\in[0,T]}|Y_t-Y^n_t|\bigg\|_{L^p(\Omega)}
+\bigg\|\sup_{t\in[0,T]}|Y^n_t-Y^{n,linear}_t|\bigg\|_{L^p(\Omega)}\\
\le& C h^{2H-\frac{1}{2}}
+C \bigg\|\sup_{t\in[0,T]}\bigg|X_t-X_{\lfloor t \rfloor_n}-\frac{t-\lfloor t \rfloor_n}{h}(X_{\lceil t \rceil ^n}-X_{\lfloor t \rfloor_n})\bigg|\bigg\|_{L^{2p}(\Omega)}\\
\le& C \bigg(h^{2H-\frac{1}{2}}+h^H\sqrt{\log \frac{T}{h}}\bigg),\quad p\ge 1,
\end{align*}
where the last inequality follows from \cite[Theorem 6]{AAP03PLE}. This result gives an answer to a conjecture in \cite{Deya} for the modified Milstein scheme defined by \eqref{step2} and \eqref{Milsteinlinear} when $H>\frac12$. Namely, the error of the modified Milstein scheme is caused by the approximation of L\'evy area with rate $2H-\frac12$ and the piecewise linear approximation of fBms with rate $h^H\sqrt{\log \frac{T}{h}}$ (see \cite{AAP03PLE}). 

\begin{remark}
	Our result indicates that the modified Milstein scheme is superior to the classical Euler method \cite{Euler08}, whose convergence rate is $2H-1$, $H\in(\frac12,1)$. Compared with the optimal convergence rate $\gamma$ of the modified Euler scheme \cite{HuEuler}, where  $\gamma=2H-\frac12$ when $H\in (\frac12,\frac34)$; $\gamma=1^-$ when $H=\frac34$; $\gamma=1$ when $H\in (\frac34,1)$, the modified Milstein scheme has higher order if $H\in [\frac34,1)$.
\end{remark}

\begin{figure}\label{f1}
	\centering
	\subfigure[$H=0.6$]{
		\begin{minipage}[t]{0.4\linewidth}
			\includegraphics[height=5cm,width=5cm]{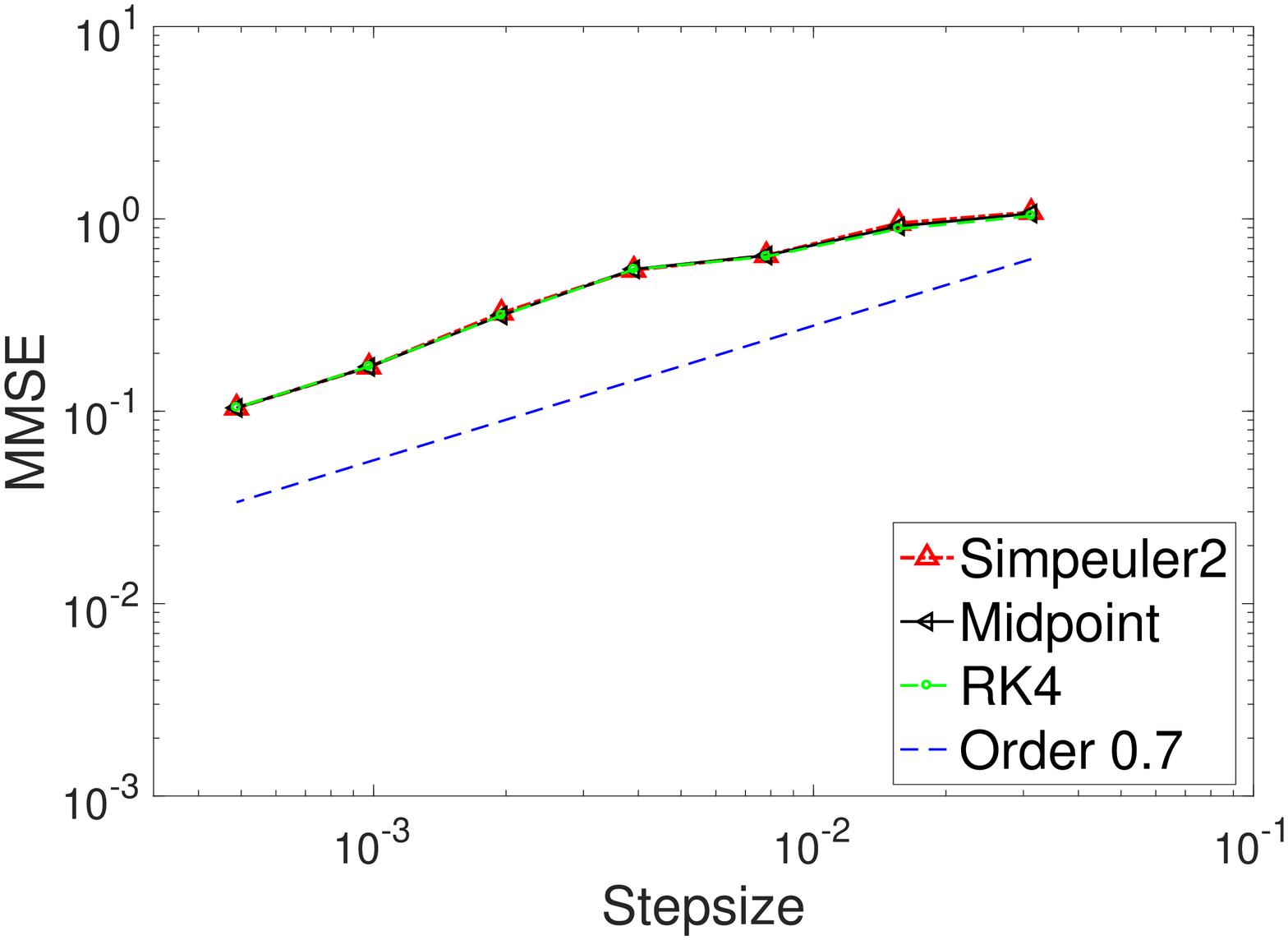}
		\end{minipage}
	}
	\subfigure[$H=0.7$]{
		\begin{minipage}[t]{0.4\linewidth}
			\includegraphics[height=5cm,width=5cm]{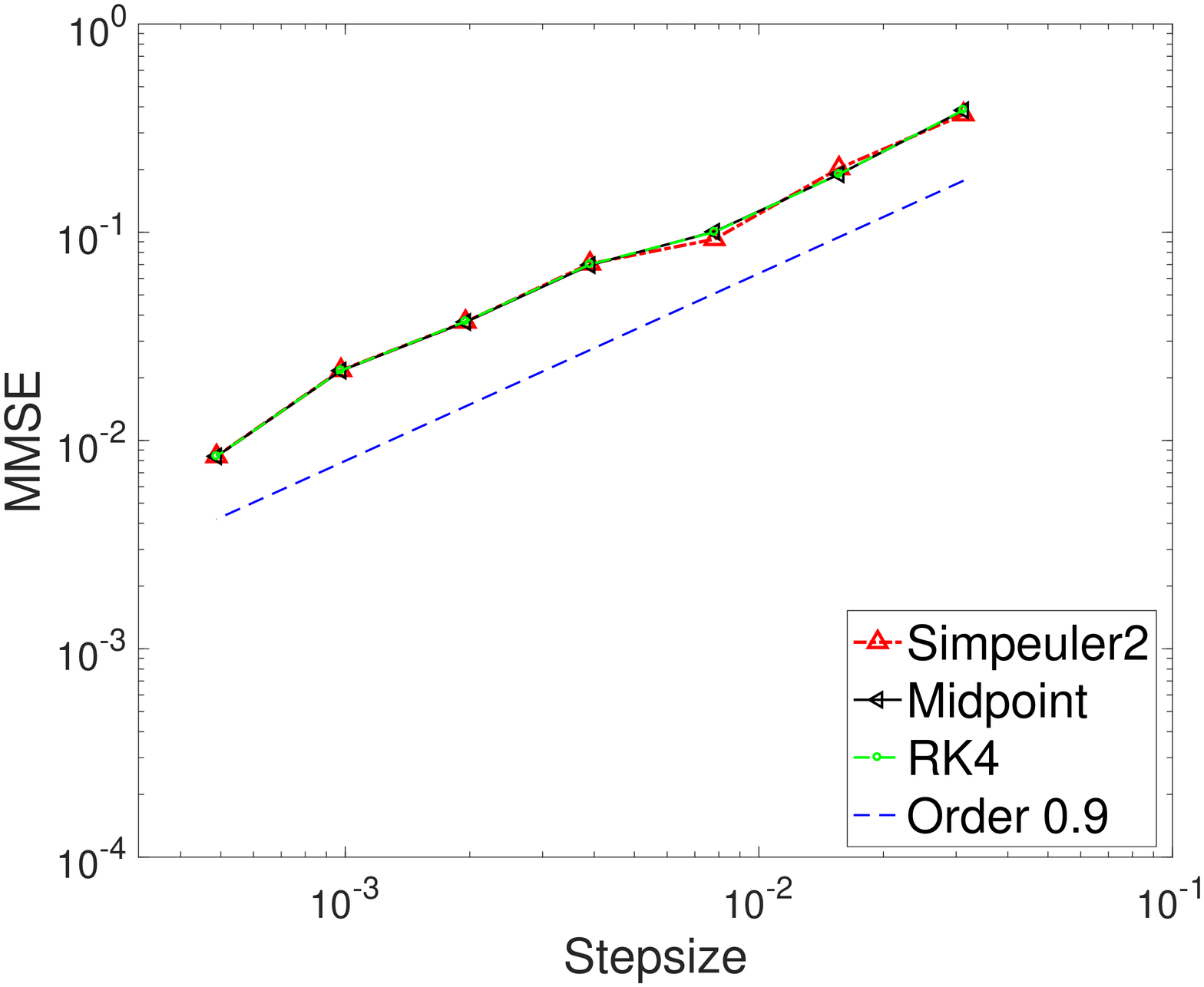}
		\end{minipage}
	}
	\subfigure[$H=0.8$]{
		\begin{minipage}[t]{0.4\linewidth}
			\includegraphics[height=5cm,width=5cm]{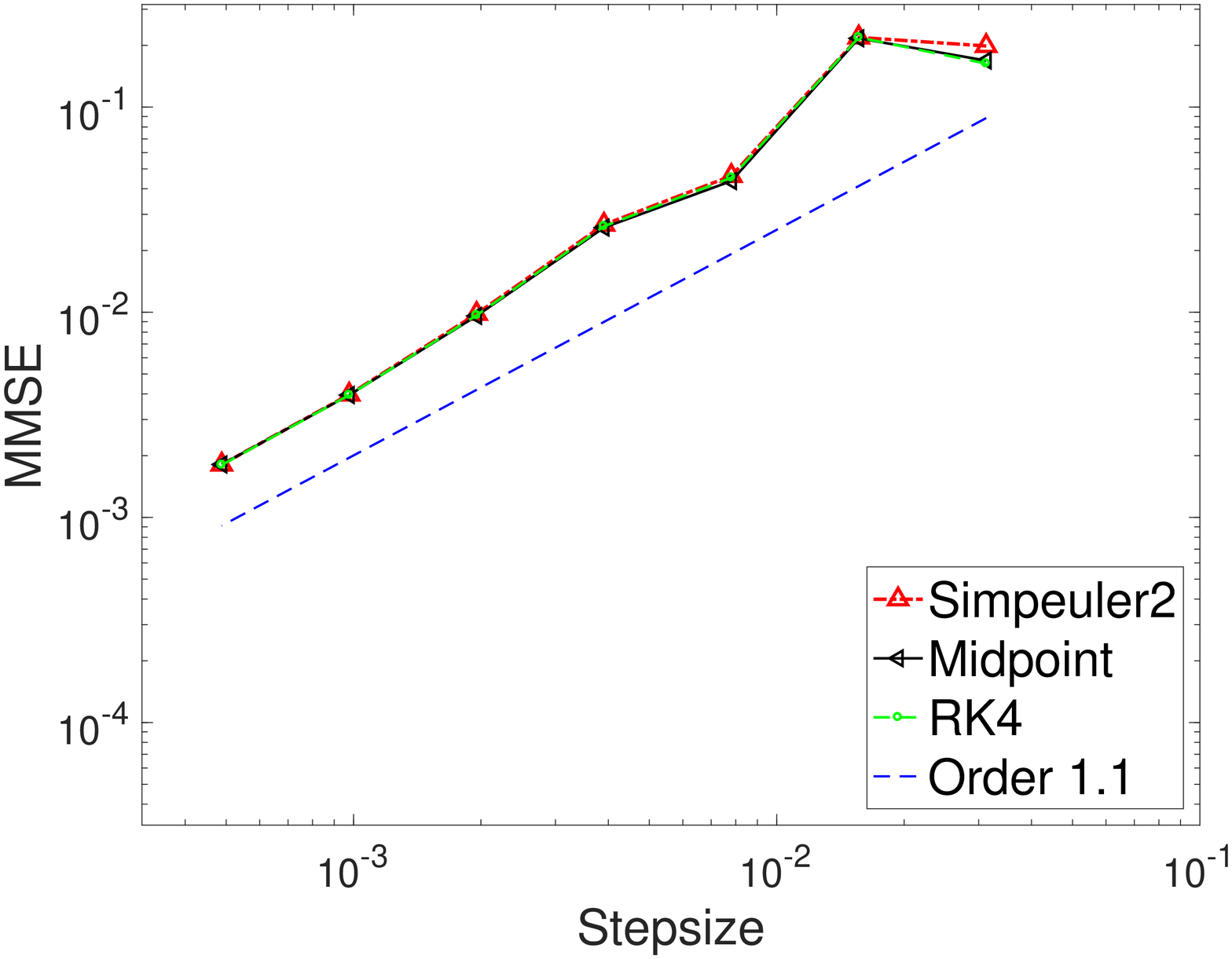}
		\end{minipage}
	}
	\subfigure[$H=0.9$]{
		\begin{minipage}[t]{0.4\linewidth}
			\includegraphics[height=5cm,width=5cm]{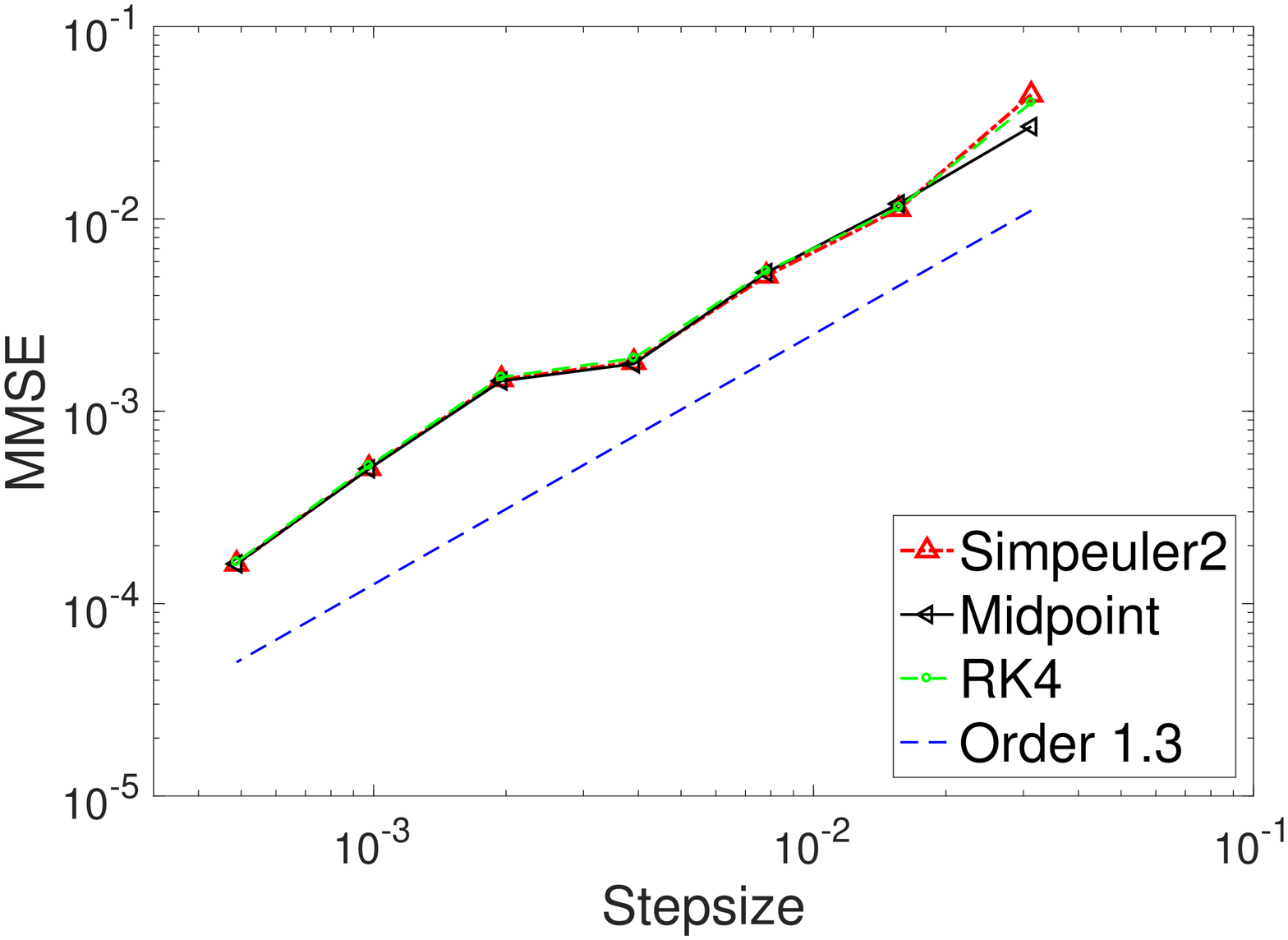}
		\end{minipage}
	}
	\caption{Maximum mean-square error (MMSE) vs. stepsize}
\end{figure}

\section{Numerical experiments}\label{sec5}
In this section, we give an example to verify our main theorems. Consider
\begin{align*}
dY_t&=3\sin(Y_t)dt+3\cos(Y_t)dX_t^2+3\sin(Y_t)dX_t^3,\quad t\in(0,1],\\
Y_0&=5,
\end{align*}
where $X^2$ and $X^3$ are independent fBms with Hurst parameter $H>\frac12$.
We compare the following three numerical schemes: simplified step-$2$ Euler scheme and two Runge--Kutta methods with coefficients expressed in the Butcher tableaus below

\begin{displaymath}
\begin{array}{c|c}
1/2&1/2\\
\hline
\ &1
\end{array},
\end{displaymath}
\begin{displaymath}
\begin{array}{c|cccc}
0&\ &\ &\ &\  \\
1/2&1/2&\ &\ &\  \\
1/2&0&1/2&\ &\  \\
1&0&0&1&\ \\
\hline
\ &1/6&2/6&2/6&1/6
\end{array}.
\end{displaymath}
In other words, the first method is the implicit midpoint scheme and the second one is a $4$-stage Runge--Kutta method satisfying conditions for order $4$ in deterministic case. Both of them satisfy condition \eqref{condition}.
Theorems \ref{theorem1} and \ref{theorem2} indicate that their maximum mean-square convergence rate is $2H-\frac12$, i.e.,
\begin{align*}
\Big\|\max_{1\le k\le n}|Y_{t_k}-Y^n_{t_k}|\Big\|_{L^2(\Omega)}\le C h^{2H-\frac{1}{2}},
\end{align*}
which is consistent with numerical results in Figure \ref{f1}. 
For each scheme, we use the numerical solution with time step $h=2^{-13}$ as the approximated `exact solution' for comparison. The number of sample paths is $1000$.


\begin{remark}
	As mentioned in the introduction, the rate $2H-\frac12$ is optimal since only increments of fBms are used in the methods under study. This fact is illustrated in Figure 1 that the 4-stage Runge-Kutta method shows the same order as other ones. Therefore, Runge--Kutta methods with stage $\mathbf{s}=1,2$ and the step-$2$ Euler scheme are enough for this rate. It is still an open problem to construct numerical schemes with orders higher than $2H-\frac12$, in which case efficient simulation of iterated integrals of multi-dimensional fBms should also be taken into consideration to make schemes implementable.
\end{remark}


\bibliography{bib}
\bibliographystyle{amsplain}

\end{document}